\documentclass{article}

\usepackage[round]{natbib}
\usepackage{xcolor}
\usepackage{amsmath,amssymb,amsthm,mathtools}
\usepackage{authblk}
\usepackage{booktabs}
\usepackage{enumitem}
\usepackage[colorlinks=true,linkcolor=blue,citecolor=blue,urlcolor=blue]{hyperref}
\usepackage{bbm}
\usepackage{cleveref}

\newtheorem{theorem}{Theorem}[section]
\newtheorem{proposition}[theorem]{Proposition}
\newtheorem{lemma}[theorem]{Lemma}
\newtheorem{corollary}[theorem]{Corollary}
\theoremstyle{definition}

\newtheorem{assumption}{Assumption}
\theoremstyle{remark}
\newtheorem{remark}[theorem]{Remark}

\Crefname{assumption}{Assumption}{Assumptions}

\DeclareMathOperator{\KL}{KL}
\DeclareMathOperator{\TV}{TV}

\newcommand{\Prob}{\mathbb P}
\newcommand{\E}{\mathbb E}
\newcommand{\tmix}{t_{\mathrm{mix}}}

\newcommand{\one}{\mathbf 1}

\newcommand{\Prb}{\mathbb{P}}

\begin{document}
\title{On the Impossibility of Estimating the Average Mixing Time}

\author{Nathan A. Judd}
\author{Amedeo Roberto Esposito}

\affil{ Okinawa Institute of Science and Technology, Japan }
\date{October 7, 2026}

\maketitle

\begin{abstract}
    We establish nearly-matching converse bounds on the sample complexity of estimating the average mixing time of a Markov chain from a single stationary trajectory. Unlike the worst-case mixing time, which maximises the total-variation distance to stationarity over all initial states, the average mixing time averages this distance under the stationary distribution and can therefore be substantially smaller. Our converse bounds establish when estimation is impossible, and match the key dependencies in existing achievability bounds, up to logarithmic factors. Together, these bounds provide nearly-tight PAC guarantees for estimating the average mixing time.
\end{abstract}

\section{Introduction}

%Existing work provides achievability results for the Ergodic and Uniformly Ergodic cases. 

Let $P$ be the transition kernel of a Markov chain on state space $\cal X$ with stationary distribution $\pi$, and let $P^t(x, \cdot)$ denote its distribution after $t\in\mathbb{Z}_+$ steps when initialised at $x$. The total-variation distance
$$
d_x(t) := \left\|P^t(x,\cdot)-\pi\right\|_{\mathrm{TV}}
$$
measures convergence of the Markov chain from the initial state $x$ to stationarity, and underpins the definitions of both the worst-case and average mixing times. The distance-to-stationarity function is defined as
$$
d_{\rm max}(t) :=\max_{x\in\cal X}d_x(t) = \max_{x\in \cal X} \left\|P^t(x,\cdot)-\pi\right\|_{\mathrm{TV}},
$$
where the slowest-to-converge initialisation is used. The associated worst-case mixing time is defined by
\begin{align}
    \label{eqn: worst-case mixing time}
    t_{\rm mix}(\xi) := \min\{t\in\mathbb{Z}_+:d_{\rm max}(t)\leq \xi\}
\end{align}

The definition is worst-case because it uses the slowest-converging initial state. Thus, once $t\geq \tmix(\xi)$, the conditional distribution $t$ steps later is within a distance $\xi$ of $\pi$, regardless of the state currently occupied by the chain. This uniform decorrelation guarantee underpins {\em i.i.d}-type concentration bounds for (Markov-)dependent data \citep{Paulin2015}, whose mixing-time penalty can be interpreted, for empirical averages, as a conservative reduction in effective sample size. This motivates the literature on estimating the worst-case mixing time and related spectral quantities from a single trajectory \citep{Hsu_et_al_2019_AoAP,wolfer_19_ALT,Wolfer2021}.

However, the uniform nature of this guarantee can be unnecessarily pessimistic for a stationary trajectory. The maximum assigns full importance to the slowest initial state, even when that state has negligible stationary probability and is, thus, rarely observed. To capture the average dependence encountered along a stationary trajectory, consider the $\beta$-mixing coefficient
\begin{align*}
    %\label{eqn: beta-definition}
    \beta(t) := \sum_{x\in\cal X}\pi(x)d_x(t) = \sum_{x\in\cal X}\pi(x)\left\|P^t(x,\cdot)-\pi\right\|_{\mathrm{TV}},
\end{align*}
which averages the distance to stationarity across initial states according to their stationary probabilities. Then, define the average mixing time for all $\xi\in(0,1)$ by
\begin{align}
    \label{eqn: average mixing time}
    \tmix^\sharp(\xi):=\min\{t\in\mathbb{Z}_+:\beta(t)\leq \xi\}.
\end{align}

Naturally, $\tmix^\sharp(\xi)\leq \tmix(\xi)$ for every chain and all $\xi\in(0,1)$. While worst-case mixing requires convergence from every state, average mixing measures convergence from a stationary-random initial state. When convergence varies substantially across initial states, particularly when slowly mixing states have little stationary mass, the average mixing time can be significantly smaller than its worst-case counterpart \citep{MunchSalez23,wolfer2026optimisticestimationconvergencemarkov}.

The $\beta$-mixing coefficient has been studied independently from the average mixing time and measures temporal dependence in stationary Markov chains. Bounds on $\beta(t)$ yield deviation guarantees for empirical averages $m^{-1}\sum_{s=1}^m f(X_s)$ around their stationary expectations $\pi(f)$, across broad classes of functions $f$ \citep{wolfer2026optimisticestimationconvergencemarkov}. Thus our results are relevant in modern computational statistics and statistical modelling, including summarising MCMC and hidden Markov models. Our converses quantify the difficulty of learning this dependence from a single, stationary trajectory. Removing the stationary assumption, and adding an initialisation penalty, akin to Appendix A in \cite{wolfer2026optimisticestimationconvergencemarkov} remains open.

%The average dependence measured by $\beta(t)$ also has a direct statistical interpretation. Many estimators based on a Markov chain trajectory are empirical averages of the form $m^{-1}\sum_{s=1}^m f(X_s)$, and are used to estimate the stationary expectation $\pi(f)$ of a function $f$. The corresponding sum is called an additive functional of the chain. Therefore, the bounds presented here apply to general estimation problems in Markov chains, not just estimating the average mixing time itself. Bounds on the $\beta$-mixing coefficients control how temporal dependence affects the deviation of such empirical averages from $\pi(f)$ \citep{wolfer2026optimisticestimationconvergencemarkov}. Thus, although average mixing is weaker than worst-case mixing, it remains a meaningful measure of the dependence relevant to estimation from stationary Markov chain data.

Writing $\tmix:=\tmix(1/4)$, the function $d_{\rm max}(t)$ satisfies \(d_{\rm max}(s+t)\leq2d_{\rm max}(s)d_{\rm max}(t)\) \citep[Lemma 4.10 and Exercise 4.3]{levin2017markov}, which underpins the useful bound $\tmix(\xi)\leq\lceil\log_2(1/\xi)\rceil\tmix$ \citep[Equation 4.34]{levin2017markov}.
%Writing $\tmix:=\tmix(1/4)$, since \(d_{\rm max}(\tmix)\leq1/4\), iteration gives \(d_{\rm max}(k\tmix)\leq2^{-k}\) for every integer \(k\in\mathbb{N}\). Taking \(k=\lceil\log_2(1/\xi)\rceil\) therefore gives
This is why $\tmix$ is often fixed by convention at the threshold $1/4$. The $\beta$-mixing coefficient does not satisfy an analogous property \citep{wolfer2026optimisticestimationconvergencemarkov}, so the threshold $\xi$ must remain explicit for the average mixing time and necessitates the use of a threshold-explicit loss function. Define the interval $I_P(\xi,\epsilon):=[\tmix^\sharp(\xi(1+\epsilon)),\tmix^\sharp(\xi(1-\epsilon))]$, where $\epsilon$ represents a multiplicative tolerance in estimating the $\beta$-mixing coefficient at the threshold $\xi$. We employ the loss function
\begin{align}
\label{eqn:loss}
\mathcal{L}(\xi,\epsilon) := 
\mathbbm{1}\big(\widehat {\tmix^\sharp}(\xi)\notin I_P(\xi,\epsilon)\big).
\end{align}
Thus the loss is zero when the estimated mixing time lies inside the interval $I_P(\xi,\epsilon)$.

Define $P^{s}(x,y)$ to be the probability of transitioning from state $x$ to $y$ in $s$ steps. Then, the law $Q_P^{(s)}(x,y):=\pi(x)P^s(x,y)$ represents the joint probability of transitioning from $x$ to $y$ in $s$ steps, under stationarity. Now, define the half-norm of $Q_P^{(s)}$:
\begin{align}
 \mathcal J_\infty^{(s)}(P)
 &:=\left(\sum_{x,y}\sqrt{Q_P^{(s)}(x,y)}\right)^2.
 \label{eq:lag-pair-effective-support}
\end{align}
Further define the supremum of this object over lags before the slower mixing time as,
\begin{equation}
 \mathcal J_{\infty,\rm env}(P)
 :=\sup_{1\leq s\leq t_{\mathrm{mix},P}^{\sharp}(\xi(1-\epsilon))}
 \mathcal J_\infty^{(s)}(P),
 \label{eq:effective-support-envelope}
\end{equation}
where we suppress the chain dependence when clear.

$\mathcal{J}_{\infty,\rm env}$ represents the effective-support of the stationary $s$-skipped transition law, and appears in the existing achievability result presented below. In particular, \cite{wolfer2026optimisticestimationconvergencemarkov} show that for uniformly ergodic Markov chains \begin{align}
\label{eqn:achievability-bound-wolfer}
\widetilde{O}\bigg(1+\frac{\tmix}{(\xi\epsilon)^2}\max\big\{C_{\rm u.e.}^2\mathcal J_{\infty,\rm env},576\log\frac{4\tmix}{\delta}\}\bigg)
\end{align}
samples suffice to guarantee accurate estimation of $\tmix^\sharp(\xi)$ in the loss \eqref{eqn:loss} with confidence $1-\delta$, where $C_{\rm u.e.}^2$ is a universal constant (specific to the uniformly ergodic class in \cite{wolfer2026optimisticestimationconvergencemarkov}). The tilde notation suppresses logarithmic factors. Whether these sample-size guarantees reflect the statistical difficulty of the problem, rather than only the analysis of a particular estimator, requires matching lower bounds.

%in particular, a $\log_2\big(\frac{1}{\xi(1-\epsilon)}$ term before the maximum and a $\lceil\log_2(1/\xi)\rceil$ term inside the logarithmic of the second argument of the maximum. Whether these sample-size guarantees reflect the statistical difficulty of the problem, rather than only the analysis of a particular estimator, requires matching lower bounds.

%\begin{align}
%O\bigg(1&+\frac{\tmix(\xi)\log_2\big(\frac{1}{\xi(1-\epsilon)}\big)}{(\xi\epsilon)^2}\times\\
%&\max\big\{C_{\rm u. erg.}^2\mathcal J_{\infty,\rm env},576\log\frac{4\tmix(1/4)\lceil\log_2(1/\xi)\rceil}{\delta}\big\}\bigg)
%\end{align}

Our first converse relates to a constructed $2^k$-valued ($k\in \mathbb{N}$) Markov chain, inspired by \citep{Levin_Tankala_2023}, which yields a sample complexity of
\begin{align}
\label{eqn: dr sample complexity}
\widetilde \Omega\bigg(1+\frac{\tmix}{(\xi\epsilon)^2}\log\frac{1}{4\delta}\bigg),
\end{align}
where $\delta$ is the failure probability. In simple terms, \eqref{eqn: dr sample complexity} means that, ignoring constants, if the sample size is below \eqref{eqn: dr sample complexity}, then no estimator can guarantee failure probability at most $\delta$ for all admissible chains. More precisely, $\nexists~\widehat t_m^\sharp$ such that $\sup_{P\in\mathcal{P}}\Prob_P^{(m)}\{\widehat t_m^\sharp\notin I_P(\xi,\epsilon)\}\leq\delta$ holds uniformly across the class of Markov kernels $\mathcal{P}$.

Our second converse uses a metastable $D$-state ($D\geq 64$) Markov chain (with \lq\lq checkerboarding\rq\rq left-to-right region transitions), yields a sample complexity of \begin{align}
    \label{eqn: checkerboard sample complexity}
    \widetilde \Omega \bigg(1+\frac{\tmix}{\epsilon^2}\max\bigg\{\frac{\mathcal J_{\infty,\rm env}}{\log \mathcal J_{\infty,\rm env}},\log\frac{1}{4\delta}\bigg\}\bigg).
\end{align}
The former converse, \eqref{eqn: dr sample complexity}, holds uniformly over both the threshold $\xi$ and the relative accuracy $\epsilon$, whereas, the latter converse, \eqref{eqn: checkerboard sample complexity}, holds for a fixed threshold and narrow $\epsilon$ range. Hence, the converses together say much more than they do individually, and there is still more work to be done to converge on a single impossibility bound.

Taken together, and within the threshold and accuracy regimes stated above, our converse bounds recover the uniformly ergodic achievability scaling for estimating the average mixing time. In particular, the linear (worst-case) mixing-time, inverse-quadratic threshold-accuracy, effective-support, and confidence dependencies of the achievability bound are all matched up to logarithmic factors. For ergodic Markov chains, the achievability result of \citet[Theorem~3.5(i)]{wolfer2026optimisticestimationconvergencemarkov} has a core  $1/(\xi\epsilon)^4$ threshold-accuracy scaling, which is a factor of $1/(\xi\epsilon)^2$ larger than both their uniformly ergodic achievability result (Theorem 3.5(ii)) and our impossibility result (\Cref{thm:register-lower} herein). Whether this discrepancy reflects looseness in the upper bound remains open.

\section{Related Work}

Statistical estimation of the worst-case mixing time and related objects has been studied extensively and thoroughly \citep{Hsu_et_al_2019_AoAP,wolfer_19_ALT,Wolfer2021}. The average mixing time was introduced in \citep{MunchSalez23} and upper bounds on estimating it were established in the preprint \citep{wolfer2026optimisticestimationconvergencemarkov}.

For reversible chains, finite-sample estimation of the worst-case mixing time is closely connected to estimating the absolute spectral gap and the least stationary mass \citep{Hsu_et_al_2019_AoAP}. More generally, transition-kernel estimation and instance-dependent confidence intervals quantify the cost of learning many conditional distributions from one dependent path \citep{wolfer_19_ALT,Wolfer2021,wolfer2024scandi}.

\section{Definitions and notation}

Let $\mathbb N=\{1,2,\ldots\}$ be the set of positive integers. The chains constructed below have a finite or countable state space $\mathcal{X}$.

We observe a single stationary trajectory $X^{(m)}:=(X_1,\ldots,X_m)$, so \(X_1\sim\pi\) and then \(X_{i+1}\mid X_i\sim P(X_i,\cdot)\) for every $i\in\{1,2,...,m-1\}$; thus the data contain \(m-1\) transitions (with self-loops allowed).

For probability measures $\mu,\nu$ on $\mathcal X$, write $\|\mu-\nu\|_{\TV}:=\frac12\sum_{x\in\mathcal X}|\mu(x)-\nu(x)|$. For $0<p,q<1$, write
\begin{align*}
 d_{\rm KL}(p\|q)
 &:=p\log\frac{p}{q}+(1-p)\log\frac{1-p}{1-q}
    %\label{eqn:bernoulli-kl}
\end{align*}
The KL divergence between a Rademacher($u$) and a Rademacher($v$) random variable ($u,v\in [-1,1]$) is
\begin{align}
d_{\rm KL}\bigg(\frac{1+u}{2}\bigg\| \frac{1+v}{2}\bigg)&\leq\frac{(u-v)^2}{1-v^2}.
\label{eq:rademacher-kl-bound}
\end{align}

For a stationary trajectory $X^{(m)}$, with $m\geq1$,
let $\Prob_P^{(m)}$ denote its path law. Throughout, $\widehat t_m^\sharp$
denotes a possibly randomised $\mathbb N$-valued estimator based on
$X^{(m)}$.

\section{Formal results}

Our converse bounds are constructed from uniformly ergodic Markov chains and thus also apply to the class of ergodic chains.

The delayed-register construction in \Cref{thm:register-lower}, inspired by \cite{Levin_Tankala_2023}, produces statistically close chains whose $\beta$-mixing coefficients lie on opposite sides of the tolerance band $\xi(1\pm\epsilon)$ at a lag comparable to their mixing times, thereby realising a standard two-point testing bottleneck: inverse-quadratic in the tolerance-band width $\xi\epsilon$.

The rare-refresh checkerboard class underlying \Cref{cor:rr-matched-maximum} exposes an effective-support bottleneck, in which information about dependence is spread over many pairs of states observed at the relevant lags along a stationary trajectory, albeit for a fixed threshold and restricted accuracy range.

\subsection{Delayed-register chain}\label{sec:dr-chain}

For $k\in\mathbb N$, let $\mathcal X_k=\{-1,+1\}^k$. For
\(a\in[0,1)\), define the delayed-register kernel
\(P^{\mathrm{dr}}_{k,a}\) on \(\mathcal X_k\) by
\begin{equation}
 \begin{aligned}
 (x_1,\ldots,x_k)&\longmapsto(x_2,\ldots,x_k,Z),\\
 \Prob_a(Z=z\mid x)&=\frac{1+ax_1z}{2}.
 \end{aligned}
 \label{eq:register-kernel}
\end{equation}
Equivalently, \(Z=x_1E\), where \(E\) is an independent Rademacher
variable with mean \(a\). Let \(U_k\) be the uniform law on
\(\mathcal X_k\), and abbreviate
\(\beta_a:=\beta_{P^{\mathrm{dr}}_{k,a}}\). Define
\begin{align}
 B_k(a)
 &:=2^{-k-1}\sum_{j=0}^k\binom{k}{j}
 \left|(1+a)^j(1-a)^{k-j}-1\right|,
 \label{eq:Bk-definition}\\
 B_k^{-1}(b)&:=\inf\{a\in[0,1):B_k(a)\geq b\}.
 \label{eq:Bk-inverse}
\end{align}

For \(0<\xi,\epsilon\leq1/4\), construct the two chains as follows
\begin{equation}
 \begin{aligned}
 a_0&=B_k^{-1}\{\xi(1-\epsilon)\},
 &a_1&=B_k^{-1}\{\xi(1+2\epsilon)\},\\
 P_i^{\mathrm{dr}}&:=P^{\mathrm{dr}}_{k,a_i},
 &&i\in\{0,1\}.
 \end{aligned}
 \label{eq:register-pair}
\end{equation}

\begin{theorem}[Delayed-register]
\label{thm:register-lower}
There exists \(K_0\in\mathbb N\) such that, for all
\(0<\xi,\epsilon\leq1/4\), every \(k\geq K_0\), every
\(m\geq1\), all \(0<\delta<1/4\), and every estimator
\(\widehat {t}_m^\sharp\), then uniform risk at most $\delta$ implies, for each \(i\in\{0,1\}\),
\begin{equation}
 m-1\geq
 \frac{t_{\mathrm{mix},P_i^{\mathrm{dr}}}}
      {48\pi e^{34}\xi^2\epsilon^2}
 \log\frac1{4\delta}.
 \label{eq:register-rate}
\end{equation}
\end{theorem}

\Cref{thm:register-lower} captures the core estimation and confidence bottlenecks uniformly in $(\xi,\epsilon)$, but does not capture the effective-support bottleneck appearing in~\eqref{eqn:achievability-bound-wolfer}.

\subsection{Rare-refresh checkerboard chain}

Fix $\xi_0:=1/20$. The rare-refresh checkerboard has two metastable regions. Most transitions hold the current state exactly; with probability $q=1/R$, a refresh switches regions, using a checkerboard distribution from left to right and a uniform distribution from right to left. Thus $R$ controls both the worst-case and average mixing scales on this class, while the information determining the $\beta$-mixing coefficient is spread across order $D^2$ possible left-to-right source--destination pairs $(X_t,X_{t+1})$, with $X_t$ the source and $X_{t+1}$ the destination. A common-coefficient parametrisation yields the confidence term, whereas independently varying moment-matched coefficients yield the effective-support term. Consequently, at $\xi_0$ and in the stated accuracy range, uniform risk at most $\delta$ over the union of the two hard subclasses informally requires \eqref{eqn: checkerboard sample complexity} samples.  
\begin{assumption}[General rare-refresh checkerboard chain conditions]
\label{asm:1}
Let \(R\geq100\) be divisible by ten and let \(D\geq64\) be divisible
by eight.
\end{assumption}

Let \(P_i:=P_{c_i}\),
\(i\in\{0,1\}\), be the common-coefficient pair of Markov chains constructed in
Subsection~\ref{subsec:rr-confidence} and define the positive integers
\begin{equation}
 M=\frac{D^2}{32},\qquad
 \ell=2\left\lceil2\log(40M)\right\rceil.
 \label{eq:rr-size-parameters}
\end{equation}

Let $c_{\rm acc},C_{\rm size},c_{\rm inf}>0$ be the universal constants defined in Appendix \ref{app:rr-constants}.

\begin{assumption}[Varying-coefficient checkerboard conditions]
\label{asm:2}
Let \(R_0\) and \(D_0\) be sufficiently large positive integers. Assume the conditions in \Cref{asm:1} hold, \(R\geq R_0\), \(D\geq D_0\), with \(M\geq C_{\rm size}\ell^2\) and \(0<\epsilon\leq1/4\) satisfying $\epsilon\leq 20c_{\rm acc}/\ell$.
\end{assumption}
Note that \Cref{asm:2} implies $\epsilon \lesssim 10^{-6}$.

Let \(\mathcal P_0,\mathcal P_1\) be the varying-coefficient subclasses constructed in Subsection~\ref{subsec:varying-rare-refresh}. Define the hard class
\begin{equation}
 \mathcal C_{R,D,\epsilon}
 :=\{P_0,P_1\}\cup\mathcal P_0\cup\mathcal P_1.
 \label{eq:rr-hard-class}
\end{equation}
Every member \(P\in\mathcal C_{R,D,\epsilon}\) is a finite-state uniformly
ergodic Markov kernel on \(D\) states and satisfies
\begin{align}
 2R&<t_{\mathrm{mix},P}^{\sharp}(\xi_0)\leq10R,
 \label{eq:rr-theorem-target-scale}\\
 D^2(1-\xi_0(1-\epsilon))^2
 &\leq\mathcal J_{\infty,\rm env}(P)\leq D^2.
 \label{eq:rr-theorem-J-scale}
\end{align}

\begin{theorem}[Common-coefficient rare-refresh checkerboard]
\label{thm:rr-confidence-component}
Under \Cref{asm:1}, for all
\(0<\epsilon,\delta\leq1/4\), every \(m\geq1\), and
every estimator \(\widehat t_m^\sharp\), uniform risk of at most $\delta$, implies, for each \(i\in\{0,1\}\),
\begin{equation}
 m-1\geq
 \frac{t_{\mathrm{mix},P_i}}{1320\,\epsilon^2}
 \log\frac1{4\delta}.
 \label{eq:rr-confidence-component-rate}
\end{equation}
\end{theorem}

\Cref{thm:rr-confidence-component} in isolation is a less informative lower bound than \Cref{thm:register-lower}, but when combined with \Cref{thm:rr-support-component} below, to make \Cref{cor:rr-matched-maximum}, it elicits joint scaling between mixing time, relative accuracy, effective-support and confidence.

\begin{theorem}[Varying-coefficient rare-refresh checkerboard]
\label{thm:rr-support-component}
For $R,D,M,\ell$ and $0<\epsilon\leq 1/4$ satisfying \Cref{asm:2}, and every \(m\geq1\), every
\(0<\delta\leq19/40\), and every estimator \(\widehat t_m^\sharp\), uniform risk
at most $\delta$, implies
\begin{equation}
 m-1\geq \frac{5c_{\rm inf}}{8}
 \sup_{P\in\mathcal P_0\cup\mathcal P_1}
 \frac{t_{\mathrm{mix},P}\,
       \mathcal J_{\infty,\rm env}(P)}
      {\epsilon^2\log\mathcal J_{\infty,\rm env}(P)}.
 \label{eq:rr-support-component-rate}
\end{equation}
\end{theorem}

\Cref{thm:rr-support-component} isolates an effective-support bottleneck in average mixing time estimation. The two checkerboard subclasses combine to form the converse below.

\begin{corollary}[Combined converse]
\label{cor:rr-matched-maximum}
Under \Cref{asm:2}, for every admissible $\epsilon$, every
\(m\geq1\), all \(0<\delta<1/4\), and every estimator
\(\widehat t_m^\sharp\), uniform risk at most $\delta$, implies
\begin{equation}
\label{eq:rr-theorem-intrinsic-rate}
 \begin{aligned}
 m-1
 &\geq
 \sup_{P\in\mathcal C_{R,D,\epsilon}}
 \frac{t_{\mathrm{mix},P}}
      {\epsilon^2}\\
 &\quad{}\times
 \max\left\{
  \frac{5c_{\rm inf}}{8}\frac{\mathcal J_{\infty,\rm env}(P)}
       {\log\mathcal J_{\infty,\rm env}(P)},
  \frac{\log\frac1{4\delta}}{1320}
 \right\}.
 \end{aligned}
\end{equation}
\end{corollary}

\begin{remark}
\Cref{cor:rr-matched-maximum} lower-bounds the trajectory length required to estimate the average mixing time within the stated threshold band; such an estimate can inform effective-sample-size calculations for MCMC.
\end{remark}

\begin{remark}
  \label{rem:confidence-log-refinement}
  For $\delta\leq1/16$, Lemma~\ref{lem:confidence-log-refinement} permits replacing $\log(1/(4\delta))$ with $\frac13\log(4/\delta)$ in the delayed-register, common-coefficient and
  combined bounds.
  %\Cref{thm:register-lower} and~\Cref{thm:rr-confidence-component} and \Cref{cor:rr-matched-maximum}
\end{remark}

\section{Overview of techniques}\label{sec:overview-of-techniques}

The delayed register construction turns a small change in one binary transition parameter into a separation between $\beta$-mixing coefficients, after a $k$-transition delay, whereas the rare-refresh checkerboard spreads the dependence signal across many possible left-to-right source--destination pairs. Although the target is $\tmix^\sharp$, the rates are stated using $\tmix$, to better facilitate comparison with the upper bound \eqref{eqn:achievability-bound-wolfer}. Both quantities are comparable on our classes by \Cref{lem:delayed-register-mixing-scales,lem:checkerboard-mixing-scales}.

\subsection{Proof of \Cref{thm:register-lower}}
\label{subsec:register-proof}

The construction separates two roles. The register retains part of the initial state before $k$ transitions, while a weak binary dependence determines the $\beta$-mixing coefficient at the first complete flush, i.e. once every initial coordinate has been replaced.

\begin{lemma}
\label{lem:delayed-register-kernel}
For every $k\in\mathbb N$ and $a\in[0,1)$, $P^{\mathrm{dr}}_{k,a}$ is a uniformly
ergodic Markov kernel on $\mathcal X_k$ with stationary distribution $U_k$.
\end{lemma}

The parameter \(a\) is selected so that $\beta$-mixing coefficient's height lies on one side or the other of the prescribed threshold band.

\begin{lemma}
    \label{lem:dr-pre-flush-block-formulas}
    For \(0\leq t<k\),
    \begin{equation}
    \beta_a(t)\geq1-2^{-(k-t)}\geq\frac12.
    \label{eq:register-preflush}
    \end{equation}
    For every $\mathcal{B}\in\mathbb{N}$,
    \begin{equation}
    \beta_a(\mathcal{B} k)=B_k(a^\mathcal{B}).
 \label{eq:register-block-identity}
\end{equation}
\end{lemma}

\Cref{prop:finite-binomial-inverse} below is the quantitative bridge between $\beta$-mixing coefficient separation and statistical indistinguishability.
%The finite binomial inverse-modulus bound with \(C=3\) and \(A=\sqrt2(C/4+1)=7\sqrt2/4\) gives \(c_A=\sqrt{2/\pi}e^{-17}\). Henceforth, take $L:=\frac2{c_A}=\sqrt{2\pi}e^{17}$.
\begin{proposition}
    \label{prop:finite-binomial-inverse}
    Define the constant $L:=\sqrt{2\pi}e^{17}$. Then, there exists \(K_0\in\mathbb N\) such that, for \(k\geq K_0\) and \(0\leq b_0<b_1\leq3/8\), \begin{align}
    B_k^{-1}(b_1)-B_k^{-1}(b_0)
    &\leq L\frac{b_1-b_0}{\sqrt{k}},
    \label{eq:Bk-inverse-lipschitz}\\
    B_k^{-1}(b_1)&\leq L\frac{b_1}{\sqrt{k}}.
    \label{eq:Bk-inverse-size}
    \end{align}
Moreover $B_k(B_k^{-1}(b))=b$ for every $0\leq b\leq3/8$.
\end{proposition}

Proof sketch: differentiate the binomial total variation, lower-bound the resulting boundary atom using \Cref{lem:binomial-atom}, and integrate the resulting slope bound.

For the pair in~\eqref{eq:register-pair}, the inverse identity $B_k(B_k^{-1}(b))=b$ from \Cref{prop:finite-binomial-inverse} and
\eqref{eq:register-block-identity} give
\[
 \beta_{a_0}(k)=\xi(1-\epsilon),
 \qquad
 \beta_{a_1}(k)=\xi(1+2\epsilon).
\]
Equation~\eqref{eq:register-preflush} keeps both mixing coefficients above
\(\xi(1+\epsilon)\) before \(k\).  The size bound in
\eqref{eq:Bk-inverse-size}, together with
\(B_k(u)\leq u\sqrt{k}\) for \(0\leq u\leq1/2\) (by \eqref{eq:rademacher-kl-bound}, KL tensorisation and Pinsker's inequality), gives
\(\beta_{a_1}(2k)=B_k(a_1^2)\leq\xi(1-\epsilon)\) after increasing the
universal \(K_0\). This proves that the two success sets are disjoint and
establishes the time comparisons used below.

\begin{lemma}
\label{lem:delayed-register-mixing-scales}
For the constant \(K_0\) in Theorem~\ref{thm:register-lower}, let
\(0<\xi,\epsilon\leq1/4\) and \(k\geq K_0\). For both chains in \eqref{eq:register-pair},
\[
 t_{\mathrm{mix},P_i^{\mathrm{dr}}}
 \leq t_{\mathrm{mix},P_i^{\mathrm{dr}}}^{\sharp}(\xi)
 \leq2k
 \leq2t_{\mathrm{mix},P_i^{\mathrm{dr}}},
\]
where, in particular, $t_{\mathrm{mix},P_0^{\mathrm{dr}}}
=t_{\mathrm{mix},P_0^{\mathrm{dr}}}^{\sharp}(\xi)=k$.
\end{lemma}

Finally, each observed transition reveals
\(E=x_1Z\in\{-1,+1\}\), a Rademacher observation of mean \(a\).  Since the
stationary initial laws agree, the path divergence is exactly $\KL(\Prob_{P_0^{\mathrm{dr}}}^{(m)}\|
     \Prob_{P_1^{\mathrm{dr}}}^{(m)})
 =(m-1)d_{\rm KL}((1+a_0)/2\| (1+a_1)/2)$,
which, after applying the bounds in~\eqref{eq:rademacher-kl-bound} and
\eqref{eq:Bk-inverse-lipschitz}, yields
\begin{equation}
    \KL(\Prob_{P_0^{\mathrm{dr}}}^{(m)}\|
     \Prob_{P_1^{\mathrm{dr}}}^{(m)})
 \leq12L^2(m-1)\frac{\xi^2\epsilon^2}{k}.
 \label{eq:register-path-kl-bound}
\end{equation}
The two success sets are disjoint, so thresholding any successful estimator
gives a binary test. Bretagnolle--Huber and
\eqref{eq:register-path-kl-bound} show that uniform risk at most $\delta$, requires
\[
 m-1\geq\frac{k}{12L^2(\xi\epsilon)^2}
 \log\frac1{4\delta}.
\]
Lemma~\ref{lem:delayed-register-mixing-scales} gives
\(t_{\mathrm{mix},P_i^{\mathrm{dr}}}\leq2k\). Since
\(L^2=2\pi e^{34}\), this proves~\eqref{eq:register-rate}. The
confidence refinement follows from Lemma~\ref{lem:confidence-log-refinement}.

\subsection{Proof of \Cref{cor:rr-matched-maximum}}

Throughout this section, write $\eta:=\xi_0\epsilon$. Unlike the delayed register, the standard checkerboard makes many hidden
directions visible at a constant fraction of the transitions and is exactly
mixed after two steps. Here the hidden directions are the possible left-to-right source--destination transitions whose probabilities depend on the unknown coefficients. To obtain a growing target average-mixing-time scale,
the next construction holds the exact current state most of the time and
applies a checkerboard transition rule only at a rare refresh. The same generic family
has two statistical specialisations: A common coefficient in every block
gives a one-dimensional confidence problem, whereas varying coefficients, defined via a moment-matched prior, elicit an effective-support bottleneck.

% ORIGIN: main.tex, subsection ``Metastable regions and checkerboard payloads''
% through paragraph ``Role of the scale lemma''; deletion-first condensation.
\subsubsection{Metastable regions}
\label{subsec:rr-construction}

Fix \(D\geq64\) divisible by eight and an integer \(R\geq100\), and set $d:=D/2,~ M:=D^2/32$ and $q:=1/R$, where $2M$ is the number of checkerboard blocks. Let $\mathcal X=\mathsf X_{\rm L}\mathbin{\cup}\mathsf X_{\rm R}$ be the chain's state-space, comprising a left half, $\mathsf X_{\rm L}=\{L_1,\ldots,L_d\}$ and a right half $\mathsf X_{\rm R}=\{R_1,\ldots,R_d\}$. The sets \(\mathsf X_{\rm L}\) and \(\mathsf X_{\rm R}\) are genuine metastable regions: every refresh switches between them, while a nonrefresh transition holds the exact state.  Denote by $U_d=\frac1d\one\one^\top$ the uniform transition matrix on $d$ states.

Pair the source and destination labels. Their Cartesian products partition the left-to-right transition probabilities into \(2M=D^2/16\) disjoint \(2\)-by-\(2\) matrix-entry blocks \(e=(a,b)\). For a scalar vector \(\theta=(\theta_e)_{e=1}^{2M}\in(-1,1)^{2M}\), write \(\theta_{ab}\) for the coordinate of block \(e=(a,b)\), and define the full \(d\)-by-\(d\) conditional left-to-right kernel \(K_\theta\) through its block restrictions
\begin{equation}
 K_\theta\big|_e
 =\frac1d
 \begin{pmatrix}
  1+\theta_e&1-\theta_e\\
  1-\theta_e&1+\theta_e
 \end{pmatrix}.
 \label{eq:rr-payload-block}
\end{equation}
The alternating signs make this a balanced checkerboard block. A displayed \(2\)-by-\(2\) restriction is not itself stochastic, but the complete row of \(K_\theta\) sums to one. The same cancellation holds down every column.

The rare-refresh kernel is
\begin{equation}
 P_\theta=
 \begin{pmatrix}
  (1-q)I_d&qK_\theta\\
  qU_d&(1-q)I_d
 \end{pmatrix}.
 \label{eq:rr-kernel}
\end{equation}
The four $d\times d$ segments in Equation~\eqref{eq:rr-kernel} encode three transition types. With probability $1-q$, the chain stays put; otherwise, it switches regions via $K_\theta$ from left to right or $U_d$ from right to left. Thus, the exit time from either region is geometric with mean $1/q=R$, making the regions metastable.

\begin{lemma}
\label{lem:checkerboard-kernel}
$P_\theta$ defined by \eqref{eq:rr-kernel} is a Uniformly Ergodic Markov kernel.
\end{lemma}

\subsubsection{General $\beta$-mixing coefficient}
\label{subsec:rr-profile}

%Then \(r_t\) is the probability of no refresh, \(w_t\) is the probability of exactly one refresh, and \(p_t\) is a signed regional imbalance, not itself an event probability. Define
For \(t\geq1\), define
\begin{equation}
 r_t=(1-q)^t,\quad p_t=(1-2q)^t,\quad
 w_t=tq(1-q)^{t-1},
 \label{eq:rr-profile-scalars}
\end{equation}
where \(r_t\) is the probability of no refresh, \(w_t\) is the probability of exactly one refresh, and \(p_t\) is a signed regional imbalance. Then further define
\begin{equation}
\begin{aligned}
 b_t
 &:=\left[r_t+\frac{(1+p_t)/2-r_t}{d}\right]-\frac1D\\
 &=r_t\left(1-\frac1d\right)+\frac{p_t}{2d}
  =r_t\left(1-\frac2D\right)+\frac{p_t}{D},\\
 g_t(v)&:=\left(w_t|v|-\frac{p_t}{2}\right)_+,
 \qquad x_+:=\max\{x,0\}.
\end{aligned}
 \label{eq:rr-baseline-hinge}
\end{equation}
$b_t$ is the coefficient-zero TV distance from stationarity; \(g_t(v)/d\) is the additional distance of a $v$-coefficient chain ($|v|>0$) when the exactly-one-refresh perturbation, $w_t|v|$, exceeds the regional imbalance, $p_t/2$. We will refer to $g_t(v)$ as the hinge contribution to the mixing coefficient.

\begin{lemma}
\label{lem:common-average-mixing-coefficient}
For the Markov chain uniquely defined by $(P_\theta,\pi)$, satisfying \Cref{asm:1}, we have
\begin{align}
 \beta_{P_\theta}(t)
 &=b_t+\frac4{D^2}\sum_{e=1}^{2M}g_t(\theta_e),\qquad \forall t\geq 1.
 \label{eq:rr-average-profile}
\end{align}
\end{lemma}
At \(t=2R\), the coefficient-zero baseline gives \(\beta_\theta(2R)>0.12>\xi_0\). At \(t=10R\), the upper envelope gives \(\beta_\theta(10R)<2\cdot10^{-4}<\xi_0\). Total-variation contraction therefore proves \eqref{eq:rr-theorem-target-scale}.
%Similarly, \(d_\theta(R/2)>1/4\) and \(d_\theta(4R)<0.056<1/4\). 
\begin{lemma}
\label{lem:checkerboard-mixing-scales}
Under Assumption~\ref{asm:1}, every kernel $P=P_\theta$ of the form \eqref{eq:rr-kernel} satisfies
\[
 \frac R2<t_{\mathrm{mix},P}
 \leq4R
 <2t_{\mathrm{mix},P}^{\sharp}\big(\frac1{20}\big)
 \leq10t_{\mathrm{mix},P}.
\]
\end{lemma}
With $T_{\rm late}:=t_{{\rm mix}, P_\theta}^\sharp(\xi_0(1-\epsilon))$ as shorthand, we get
\begin{equation}
 \begin{aligned}
     \TV\bigl(Q_{P}^{(T_{\rm late})},\pi\otimes\pi\bigr) &= \frac{1}{2}\sum_{x,y}{\pi(x)\big|P^{(T_{\rm late})}(x,y)-\pi(y)\big|}\\
     &= \sum_x{\pi(x)\TV(P^{(T_{\rm late})}(x,\cdot),\pi)}\\
     &=\beta_{P}(T_{\rm late})\leq\xi_0-\eta.
 \end{aligned}
 \label{eq:rr-pair-TV}
\end{equation}
For probability masses \(Q,U\), the pointwise inequality
$$
\sqrt{Q(z)U(z)}\geq\min\{Q(z),U(z)\}
$$
implies \(\sum_z\sqrt{Q(z)U(z)}\geq1-\TV(Q,U)\). Apply this with \(Q=Q_{P_\theta}^{(T_{\rm late})}\) and \(U=\pi\otimes\pi\), which is uniform on the \(D^2\) ordered pairs. Together with Cauchy--Schwarz for the converse bound, this gives \eqref{eq:rr-theorem-J-scale}.

\subsubsection{Common-coefficient checkerboard}
\label{subsec:rr-confidence}

Set every one of the \(2M\) block parameters equal to the scalar \(c\), and write \(\theta(c):=(c)_{e=1}^{2M}\), \(P_c:=P_{\theta(c)}\), and \(\beta_c:=\beta_{P_c}\). Then, the specific block visited is irrelevant in estimating $c$, the family is statistically one-dimensional, despite the support of $Q_P^{(s)}(x,y)$ being of order \(D^2\). So, the $\beta$-mixing coefficient becomes
\begin{equation}
 \beta_c(t)=b_t+\frac14g_t(c).
 \label{eq:rr-common-profile}
\end{equation}
Now we show that there exist two chains that are well separated in our loss function. Take $s_{0,1}=33R/10$ and write \(r=r_{s_{0,1}}\), \(p=p_{s_{0,1}}\), \(w=w_{s_{0,1}}\), and \(b=b_{s_{0,1}}\).
\begin{lemma}
\label{lem:rr-confidence-calibration}
For all \(0<\epsilon\leq1/4\), there exist coefficients satisfying
\begin{equation}
 0<c_0<c_1<0.99,
 \qquad c_1-c_0<71\eta.
 \label{eq:rr-confidence-coefficient-range}
\end{equation}
For these coefficients,
\begin{equation}
 \beta_{c_0}(s_{0,1})=\xi_0-\eta,
 \qquad
 \beta_{c_1}(s_{0,1})=\xi_0+\frac{21}{20}\eta.
 \label{eq:rr-confidence-profile-values}
\end{equation}
Consequently, $t_{{\rm mix},P_{c_0}}^\sharp(\xi_0(1-\epsilon))\leq s_{0,1}<t_{{\rm mix},P_{c_1}}^\sharp(\xi_0(1+\epsilon))$, and the two output sets are disjoint.
\end{lemma}

The initial stationary laws agree, so the path factorisation gives \(m-1\) times the stationary-average one-step divergence. A row in \(\mathsf X_{\rm R}\) is parameter-free. In a row \(x\in\mathsf X_{\rm L}\), the self-loop is also common (to both chains), and direct summation over the two checkerboard parities gives
\[
 \KL(P_{c_0}(x,\cdot)\| P_{c_1}(x,\cdot))
 =qd_{\rm KL}((1+c_0)/2\| (1+c_1)/2).
\]
By the Markov property, stationarity, and the fact that $\pi(\mathsf X_{\rm L})=1/2$, using~\eqref{eq:rademacher-kl-bound} and
\eqref{eq:rr-confidence-coefficient-range}, we get
\begin{equation}
\KL(\Prob_{P_{c_0}}^{(m)}\|\Prob_{P_{c_1}}^{(m)})
 \leq 132000\frac{(m-1)\eta^2}{R}.
 \label{eq:rr-confidence-kl}
\end{equation}

Thresholding any estimator on the two disjoint success sets gives a test. Then Bretagnolle--Huber inequality~\citep[Lemma~2.6]{Tsybakov_2009} and~\eqref{eq:rr-confidence-kl} imply that uniform risk at most \(\delta\) requires
\begin{equation}
 m-1\geq\frac{R}{132000\eta^2}\log\frac1{4\delta}
 \label{eq:rr-confidence-delta-bound}
\end{equation}
Substituting \(\eta=\epsilon/20\) and the bound
\(t_{\mathrm{mix},P_i}\leq4R\) from
Lemma~\ref{lem:checkerboard-mixing-scales} into
\eqref{eq:rr-confidence-delta-bound} proves
\eqref{eq:rr-confidence-component-rate}. The confidence refinement follows from Lemma~\ref{lem:confidence-log-refinement}.

\subsubsection{Varying-coefficient checkerboard}
\label{subsec:varying-rare-refresh}

For the effective-support bottleneck, $M$ control blocks (with a common coefficient) centre the two prior means of the $\beta$-mixing coefficient around the threshold $\xi_0$, while independent, varying coefficients on $M$ hidden blocks spread information across many source--destination pairs. The hidden-coefficient priors have different expected hinge contributions but matching moments through degree $\ell$. Calibration and concentration turn the hinge gap into separation of the $\beta$-mixing coefficient at the chosen time-lag, while moment matching cancels the low-order terms in the Poissonised likelihood comparison.

Concretely, the standard moment-matching construction \citep{Fuzzy_Le_Cam_1,Fuzzy_Le_Cam_2} gives symmetric laws $\nu_0,\nu_1$ on $[-1,1]$ whose moments through the even degree $\ell$ agree and whose oriented positive-part gap $\gamma_\ell$ satisfies
\begin{equation}
 \E_{\nu_1}Z_+-\E_{\nu_0}Z_+=\gamma_\ell,
 \qquad \gamma_\ell\geq\frac{a_\star}{\ell}.
 \label{eq:loc-unit-hinge-gap}
\end{equation}
At the integer time-lag \(s_{0,2}=16R/5\), abbreviate \(p=p_{s_{0,2}}\), \(w=w_{s_{0,2}}\), \(b=b_{s_{0,2}}\), and \(\tau=p/(2w)\). Define
\begin{equation}
 h:=\frac{32\eta}{w\gamma_\ell},
 \qquad V:=hZ,
 \qquad U:=\tau+V.
 \label{eq:loc-translated-prior}
\end{equation}
Write $\mu_i$ for the law of $U$ when $Z\sim\nu_i$.
\begin{lemma}[Hidden-coefficient support]
\label{lem:loc-hidden-coefficient-support}
Under \Cref{asm:2}, for both $i\in\{0,1\}$, $\mu_i$ is supported on
$[\tau/2,3\tau/2]\subset(0,3/200)$.
\end{lemma}
Recall the definition of $g$ from \eqref{eq:rr-baseline-hinge}, then define
\begin{equation}
 \bar g_i:=\E_{\mu_i}g_{s_{0,2}}(U).
 \label{eq:loc-hinge-means}
\end{equation}
By Lemma~\ref{lem:loc-hidden-coefficient-support}, $U>0$. Since $w\tau=p/2$, the hinge contribution, $g_t(v)$, linearises exactly:
\begin{equation}
 \begin{aligned}
 g_{s_{0,2}}(\tau+hZ)&=whZ_+,\\
 \bar g_1-\bar g_0&=32\eta.
 \end{aligned}
 \label{eq:loc-exact-hinge}
\end{equation}
The above construction for an admissible hidden coefficient, $u$, via symmetry of $Z$, converts the absolute-first-moment separation of $Z$ into the exact separation of the hinge term in~\eqref{eq:loc-exact-hinge}, which then creates separation in the $\beta$-mixing coefficient.

For \(u=(u_1,\ldots,u_M)\), let \(\theta(c,u)\) assign \(c\) to every control block and \(u_e\) to hidden block \(e\), and abbreviate \(P_{(c,u)}:=P_{\theta(c,u)}\) and \(\beta_{(c,u)}:=\beta_{P_{(c,u)}}\). At the chosen lag, $s_{0,2}$,
\begin{equation}
 \beta_{(c,u)}(s_{0,2})
 =b+\frac18g_{s_{0,2}}(c)
 +\frac1{8M}\sum_{e=1}^M g_{s_{0,2}}(u_e).
 \label{eq:loc-split-profile}
\end{equation}
The hidden blocks therefore have total weight \(1/8\) in \(\beta_{(c,u)}(s_{0,2})\), so their expected contributions under the two priors differ by \((\bar g_1-\bar g_0)/8=4\eta\).
Calibrate the controls by
\begin{equation}
 g_{s_{0,2}}(c)
 =8(\xi_0-b)-\frac{\bar g_0+\bar g_1}{2}.
 \label{eq:loc-control-calibration}
\end{equation}
Under hypothesis \(i\), let \(\boldsymbol U=(U_1,\ldots,U_M)\sim\mu_i^{\otimes M}\), and write \(\E_i\) for expectation under this product law. The calibration centres the two prior means at \(\xi_0\). Indeed, taking expectations in~\eqref{eq:loc-split-profile} and inserting~\eqref{eq:loc-control-calibration}, then using \(\bar g_1-\bar g_0=32\eta\), gives
\[
 \E_0\beta_{(c,\boldsymbol U)}(s_{0,2})=\xi_0-2\eta,
 \qquad
 \E_1\beta_{(c,\boldsymbol U)}(s_{0,2})=\xi_0+2\eta.
\]

\begin{lemma}
\label{lem:coefficient-validity-check}
    $c,u_e\in (-1,1)~\forall ~e\in \{1,2,...,M\}$.
\end{lemma}
    
For \(i\in\{0,1\}\), define
\begin{align*}
    \mathcal{U}_0 &:= \bigg\{u\in \rm{supp}(\mu_0)^M : \frac{1}{M} \sum_{e=1}^M{g_{s_{0,2}}(u_e)}\leq \bar{g}_0 + 8\eta\bigg\}\\
    \mathcal{U}_1 &:= \bigg\{u\in \rm{supp}(\mu_1)^M : \frac{1}{M} \sum_{e=1}^M{g_{s_{0,2}}(u_e)} > \bar{g}_1 - 8\eta\bigg\}.
\end{align*}
Let $\mathcal{P}_i := \{P_{(c,u)}:u\in \mathcal{U}_i\},\quad i\in\{0,1\}$ be the classes of checkerboard kernels with coefficients belonging to $\mathcal{U}_i$, defined above.

\begin{lemma}[Prior concentration and separation]
  \label{lem:loc-profile-concentration}
  For \(i\in\{0,1\}\), we have $\Prb_{\mu_i^{\otimes M}}(\mathcal{U}_i)\geq 79/80$. Moreover, every \(P\in\mathcal P_0\) satisfies \(\beta_P(s_{0,2})\leq(1-\epsilon)/20\), whereas every \(P\in\mathcal P_1\) satisfies \(\beta_P(s_{0,2})>(1+\epsilon)/20\).
\end{lemma}

We next expose the statistical experiment generated by the hidden blocks. Represent each transition opportunity by an independent Bernoulli-\(q\) refresh flag. For a coefficient vector \(u\), let \(\widetilde Q_{c,u}(i,j):=K_{\theta(c,u)}(i,j)/d\) be the law of a left-source, right-destination pair. Pre-draw iid \(\widetilde Q_{c,u}\) pairs for left-to-right refreshes. Add an independent synthetic \(\widetilde Q_{c,u}\) pair at every right-to-left refresh for a coupled chain. At every left-to-right refresh, the source label is uniform, either by initialisation at stationarity or by the uniform right-to-left refresh. (See Section 4 in \cite{levin2017markov} for an introduction to probabilistic coupling.) Each resulting mark selects one of the \(2M\) blocks uniformly and, conditional on hidden block \(e\), has parity $\Prob(S=\pm 1\mid e,u_e)=\frac{1\pm u_e}{2}$. Together with the ancillary uniforms, these flags and marks determine the original path by parameter-free post-processing, so the marked experiment is more informative.
% readability // intuition // rewrite long sentences. // proof plans // sketches

Let $N^\circ\sim\operatorname{Poi}\{2(m-1)\}$ be the number of transition opportunities. Conditional on the full coefficient vector, independently mark each opportunity by no refresh or by its refresh block and parity. The Poisson marking and restriction theorems \citep{last2017lectures} imply that the counts in these disjoint categories are independent. A specified hidden block is selected with probability $q/(2M)$ per opportunity, so its total rate is $\lambda:=(m-1)q/M$; conditional on its coefficient $U$, splitting by parity gives the independent counts
\begin{align*}
 X_+&\sim\operatorname{Poi}\left\{\lambda\frac{(1+U)}{2}\right\},\quad X_-\sim\operatorname{Poi}\left\{\lambda\frac{(1-U)}{2}\right\}.
\end{align*}
Conditional on the counts $(X_{e,+},X_{e,-})_{e\in\{1,...,M\}}$, all discarded data have the same law under both priors: the no-refresh and control counts have their common Poisson laws, the cells within each parity and the ordering can be sampled uniformly, and the initial region and ancillary variables have common parameter-free laws. Hence a single reconstruction kernel recovers the Poissonised marked experiment from the count vector; retaining only the counts loses no information about the hypothesis.
Let \(F_u\) denote the product-Poisson block law of $(X_+,X_-)$ at coefficient \(u\), and put
\(\Lambda_{\tau+v}:=dF_{\tau+v}/dF_\tau\). The shifted cross-moment identity
\begin{equation}
 \E_{F_\tau}\{\Lambda_{\tau+v}\Lambda_{\tau+z}\}
 =\exp\left\{\frac{\lambda vz}{1-\tau^2}\right\},
 \label{eq:loc-shifted-poisson-cross}
\end{equation}
allows us to compare the two likelihood ratios using moments.
Let $\overline\Lambda_i:=\E_{V\sim h\nu_i}\Lambda_{\tau+V}$. Applying the power-series expansion to \eqref{eq:loc-shifted-poisson-cross}, the matching moments cancel all terms of degree at most $\ell$. Since $|V|\leq h$ and using the triangle inequality, the $L^2(F_\tau)$ distance between the two single-block mixture likelihood ratios satisfies
  \begin{align*}
   \|\overline\Lambda_1-\overline\Lambda_0\|_{L^2(F_\tau)}^2
   \leq4\sum_{j=\ell+1}^\infty\frac{(Ah^2)^j}{j!},
   \qquad A:=\frac{\lambda}{1-\tau^2}.
  \end{align*}
For trajectory lengths satisfying $m-1\leq\frac{c_{\rm inf}M}{q\eta^2\ell}$, the choice of \(c_{\rm inf}\), Lemma~\ref{lem:rr-localised-lag-bounds}, and Stirling's elementary lower bound give
\begin{equation}
 Ah^2\leq\frac{\ell}{32e},\qquad
 \sum_{j=\ell+1}^\infty\frac{(Ah^2)^j}{j!}
 \leq2\,32^{-(\ell+1)}.
 \label{eq:loc-factorial-tail-spine}
\end{equation}
Let \(\overline{\mathbb P}^{\rm Poi}_i\) and \(\overline{\mathbb P}^{(m)}_i\) denote the prior-\(i\) Poissonised and fixed-size marked laws, respectively.
Cauchy--Schwarz and tensorisation over the \(M\) independent hidden blocks,
before conditioning, thus give
\begin{equation}
 \TV(\overline{\mathbb P}^{\rm Poi}_0,
     \overline{\mathbb P}^{\rm Poi}_1)
 \leq\sqrt2\,M\,32^{-(\ell+1)/2}<\frac1{80}.
 \label{eq:loc-product-TV-spine}
\end{equation}
%De-Poissonisation costs a factor smaller than two when \(m-1\geq3\): The (parameter-free) post-processing rule is that if $N^\circ\geq m-1$ then keep the first $m-1$ transition opportunities, otherwise return a dummy variable. Then, since the dummy variable outcome is equally likely under the two priors, the overall TV distance simplifies to $\Prb(N^\circ\geq m-1){\rm TV}(\bar {\mathbb P}_0^{m}, \bar {\mathbb P}_1^{m})$. Then $\Prb(N^\circ\geq m-1)$ is precisely the inverse-multiplicative cost of de-Poissonisation. The lower bound 2.13(b) in \cite{vershynin2018high} yields the upper bound  $\Prb(N^\circ < m-1)\leq \exp(-(m-1)/4)$. Then, $1/\Prb(N^\circ \geq m-1)\leq 2$ for every $m-1\geq 3$ and thus de-Poissonisation costs a factor of $2$ in the upper bound, and the fixed-size marked laws have total variation below \(1/40\). Conditioning on the two separating events adds at most \(1/80\) on each side, leaving total variation below \(1/20\).
De-Poissonisation costs a factor smaller than two for every \(m\geq 1\). Consider the parameter-free post-processing rule that retains the first \(m-1\) transitions when \(N^\circ\geq m-1\), and otherwise outputs a common dummy symbol \(\dagger\). Since \(N^\circ\) is independent of the transition sequence, the processed law under prior \(i\) is
$$
\Prb(N^\circ\geq m-1)\overline{\mathbb P}^{(m)}_i + \Prb(N^\circ<m-1)\delta_\dagger.
$$
The common dummy components cancel in total variation, so the data processing inequality implies
$$
\Prb(N^\circ\geq m-1)
   \TV\!\left(\overline{\mathbb P}^{(m)}_0,
              \overline{\mathbb P}^{(m)}_1\right)
   \leq
   \TV\!\left(\overline{\mathbb P}^{\rm Poi}_0,
              \overline{\mathbb P}^{\rm Poi}_1\right).
$$
Moreover, the Poisson lower-tail bound (Exercise 2.13(b) in \citep{vershynin2018high}) gives
$$
\Prb(N^\circ<m-1) \leq e^{-(1-\log 2)(m-1)} \leq e^{-(m-1)/4}.
$$
Thus \(\Prb(N^\circ\geq m-1)>1/2\) for \(m-1\geq3\), and
$
\TV(\overline{\mathbb P}^{(m)}_0,
              \overline{\mathbb P}^{(m)}_1)<\frac1{40}.
$ For $m=1,2,3$, these are $1$, $1-e^{-2}$, and $1-5e^{-4}$, all of which are above $1/2$.

Conditioning on the two separating events adds at most \(1/80\) on each side, leaving total variation below \(1/20\). Then, by Le Cam's inequality, whenever $m-1\leq c_{\rm inf}M/(q\eta^2\ell)$, there exists at least one chain with failure probability greater than $19/40$. Hence, any estimation satisfying uniform risk at most $\delta$, with $\delta\leq 19/40$, must have $m-1> c_{\rm inf}M/(q\eta^2\ell)$. Since $q=1/R$ and $\eta=\epsilon/20$, Lemma~\ref{lem:checkerboard-mixing-scales} gives $t_{\mathrm{mix},P}\leq4R$, while $\mathcal J_{\infty,\rm env}(P)\leq 32M$ and $\ell<5\log\mathcal J_{\infty,\rm env}(P)$. The identity $400/(4\cdot32\cdot5)=5/8$ then proves Theorem~\ref{thm:rr-support-component}.

\subsection{Combining the two matched subclasses}
\label{subsec:loc-combination}

Under Assumption~\ref{asm:2}, which includes Assumption~\ref{asm:1},
\eqref{eq:rr-confidence-coefficient-range} and
Lemma~\ref{lem:coefficient-validity-check} make every component
of~\eqref{eq:rr-hard-class} an admissible instance of the rare-refresh
family. Hence the generic calculation in Subsection~\ref{subsec:rr-profile} proves the common scale
bounds~\eqref{eq:rr-theorem-target-scale}--\eqref{eq:rr-theorem-J-scale}
throughout \(\mathcal C_{R,D,\epsilon}\).

Uniform risk on this unified class therefore induces both hard subclasses simultaneously. Since
\(0<\delta<1/4<19/40\), the component proofs give
\(m-1\geq R\log\{1/(4\delta)\}/(132000\eta^2)\) and
\(m-1>c_{\rm inf}RM/(\eta^2\ell)\), respectively. Recalling that
$\eta=\xi_0\epsilon=\epsilon/20$, the bound
$t_{\mathrm{mix},P}\leq4R$ from
Lemma~\ref{lem:checkerboard-mixing-scales}, together with
\eqref{eq:rr-theorem-J-scale}, \(M=D^2/32\),
\(\ell<5\log(D^2)\), and the monotonicity of
\(x/\log x\) for \(x>e\), show that these construction-scale bounds
dominate, respectively, the confidence and effective-support terms
in~\eqref{eq:rr-theorem-intrinsic-rate} for every
\(P\in\mathcal C_{R,D,\epsilon}\). Taking their pointwise maximum and then
the supremum proves the result. The confidence refinement follows from
Lemma~\ref{lem:confidence-log-refinement}.

\section{Discussion}

The constructions isolate uniform threshold--accuracy and effective-support bottlenecks. Open problems are to make the latter bottleneck uniform over threshold and accuracy, close the logarithmic gaps with \citep[Theorem~3.5(ii)]{wolfer2026optimisticestimationconvergencemarkov}, address non-stationary initialisations, and determine whether the ergodic-chain $(\xi\epsilon)^{-2}$ gap is intrinsic or an artefact of the upper bound \citep[Theorem~3.5(i)]{wolfer2026optimisticestimationconvergencemarkov}.

%
%$\log_2(1/(\xi(1-\epsilon)))$ and $\log\log_2(1/\xi)$

\bibliography{refs}

\clearpage
\appendix
\thispagestyle{empty}

\section{Proofs of Lemmas and Propositions}

\begin{proof}[Proof of \Cref{lem:delayed-register-kernel}]
For $x=(x_1,\ldots,x_k)\in\mathcal X_k$, summing over its two possible
successors gives
\[
 \sum_{z\in\{-1,+1\}}\frac{1+ax_1z}{2}=1,
\]
so $P^{\mathrm{dr}}_{k,a}$ is a Markov kernel. For
$y=(y_1,\ldots,y_k)$, its only possible predecessors are
$x^{(\sigma)}=(\sigma,y_1,\ldots,y_{k-1})$, $\sigma\in\{-1,+1\}$, and
\[
 \sum_{\sigma\in\{-1,+1\}}P^{\mathrm{dr}}_{k,a}(x^{(\sigma)},y)
 =\sum_{\sigma\in\{-1,+1\}}\frac{1+a\sigma y_k}{2}=1.
\]
Thus $P^{\mathrm{dr}}_{k,a}$ is doubly stochastic and $U_k$ is stationary.
Since $a<1$, each prescribed appended sign has conditional probability at
least $(1-a)/2>0$. Hence any $y$ can be reached from any $x$ in $k$ steps by
successively appending $y_1,\ldots,y_k$, which occurs with probability $[(1-a)/2]^k$, so the chain is irreducible. Finally,
for $\mathbf 1=(1,\ldots,1)$,
\[
 P^{\mathrm{dr}}_{k,a}(\mathbf 1,\mathbf 1)=\frac{1+a}{2}>0,
\]
so the chain is aperiodic. A finite irreducible and aperiodic chain is
uniformly ergodic.
\end{proof}

\begin{proof}[Proof of \Cref{lem:dr-pre-flush-block-formulas}]
Fix an initial word $x=(x_1,\ldots,x_k)$. After $t\leq k$ transitions,
the register has the form
\[
 (x_{t+1},\ldots,x_k,x_1E_1,\ldots,x_tE_t),
\]
where $E_1,\ldots,E_t$ are independent Rademacher variables of mean $a$.
For $t<k$, its law is therefore supported on
\[
 S_{x,t}:=\{y\in\mathcal X_k:(y_1,\ldots,y_{k-t})
       =(x_{t+1},\ldots,x_k)\}.
\]
The final $t$ coordinates are free, so $|S_{x,t}|=2^t$ and
\[
 (P^{\mathrm{dr}}_{k,a})^t(x,S_{x,t})=1,
 \qquad U_k(S_{x,t})=\frac{2^t}{2^k}=2^{-(k-t)}.
\]
The event characterisation of total variation gives
\[
 \left\|(P^{\mathrm{dr}}_{k,a})^t(x,\cdot)-U_k\right\|_{\TV}
 \geq1-2^{-(k-t)}.
\]
This holds for every $x$, so averaging over $x\sim U_k$ yields the first
inequality in~\eqref{eq:register-preflush}, since
\begin{align*}
    \beta_a(t)&=2^{-k}\sum_{x\in\mathcal X_k}
 \left\|(P^{\mathrm{dr}}_{k,a})^t(x,\cdot)-U_k\right\|_{\TV} \geq2^{-k}\sum_{x\in\mathcal X_k}\{1-2^{-(k-t)}\}
 =1-2^{-(k-t)}.
\end{align*}
Since $t<k$ implies $k-t\geq1$, its right-hand side is at least $1/2$.

For the block identity, index the fresh multipliers by coordinate and block.
After $\mathcal B$ complete blocks, the register is
\[
 (x_1W_1,\ldots,x_kW_k),
 \qquad W_i:=\prod_{r=1}^{\mathcal B}E_{i,r},
\]
where all the $E_{i,r}$ are independent Rademacher variables of mean $a$.
The products defining distinct $W_i$ use non-overlapping collections of these
independent multipliers, and
\[
 \E W_i=\prod_{r=1}^{\mathcal B}\E E_{i,r}=a^{\mathcal B}.
\]
Since each $W_i$ is $\{-1,+1\}$-valued, $W_1,\ldots,W_k$ are independent
Rademacher of mean $a^{\mathcal B}$.

Let $T_x(w):=x\odot w$ and let $\mu_{\mathcal B}$ be the law of
$(W_1,\ldots,W_k)$. Since $T_x(T_x(w))=w$, the map $T_x$ is its own inverse
and hence a bijection of $\mathcal X_k$. Therefore, for every
$A\subseteq\mathcal X_k$,
\[
 U_k(T_x^{-1}A)=\frac{|T_x^{-1}A|}{2^k}=\frac{|A|}{2^k}=U_k(A).
\]
Moreover, as $A$ ranges over all subsets of $\mathcal X_k$, so does
$T_x^{-1}A$. The event characterisation of total variation therefore gives
\begin{align*}
    \|\mu_{\mathcal B}\circ T_x^{-1}-U_k\|_{\TV}
 &=\sup_A|\mu_{\mathcal B}(T_x^{-1}A)-U_k(T_x^{-1}A)| =\|\mu_{\mathcal B}-U_k\|_{\TV}.
\end{align*}

Thus we may remove the fixed signs $x_i$. A word with $j$ plus signs then
has probability
\[
 2^{-k}(1+a^{\mathcal B})^j(1-a^{\mathcal B})^{k-j},
\]
and there are $\binom kj$ such words. The resulting distance is independent
of $x$, so its average over $x\sim U_k$ is $\beta_a(\mathcal B k)$. The
finite-space formula for total variation consequently gives
\begin{align*}
 \beta_a(\mathcal B k)
 &=\frac12\sum_{j=0}^k\binom kj
   \left|2^{-k}(1+a^{\mathcal B})^j
   (1-a^{\mathcal B})^{k-j}-2^{-k}\right| =B_k(a^{\mathcal B}),
\end{align*}
which proves~\eqref{eq:register-block-identity}.
\end{proof}

\begin{proof}[Proof of \Cref{prop:finite-binomial-inverse}]
Use the notation of \Cref{lem:h-equivalence}.  That lemma gives the
binomial upper-tail representation~\eqref{eq:Bk-binomial-tail}.  By
\Cref{lem:binomial-equivalent-derivative}, $B_k$ is locally absolutely
continuous and~\eqref{eq:Bk-derivative} holds almost everywhere.  The goal
is to obtain a uniform lower bound of order $\sqrt{k}$
for \(B_k'(a)\) on a range containing \(B_k^{-1}([0,3/8])\).  Integrating
that derivative bound will then control increments of \(B_k^{-1}\).

We next show that the inverse images of \([0,3/8]\) occur where \(a\) is of
order \(1/\sqrt{k}\).  Take \(C=3\) and set \(a=C/\sqrt{k}\).  Under
\(P_a\),
\[
 \E_aS=ka,
 \qquad
 \operatorname{Var}_a(S)=k(1-a^2).
\]
Cantelli's inequality therefore gives
\begin{align*}
 P_a(S\leq0)
 &=P_a(S-ka\leq-ka) \leq\frac{k(1-a^2)}{k(1-a^2)+k^2a^2} =\frac{1-a^2}{1+(k-1)a^2}.
\end{align*}
Substituting \(a=C/\sqrt{k}\) yields
\[
 P_a(S\leq0)
 \leq
 \frac{1-C^2/k}{1+C^2(1-1/k)}
 <\frac{1}{1+C^2}=\frac1{10}.
\]
Hence \(P_a(S>0)>9/10\), whereas symmetry gives
\(P_0(S>0)\leq1/2\).  Since total variation dominates the difference in
probabilities of any fixed event,
\[
 B_k(C/\sqrt{k})
 \geq P_a(S>0)-P_0(S>0)
 >\frac9{10}-\frac12
 =\frac25>\frac38.
\]
Together with \(B_k(0)=0\), this shows that every generalised inverse of a
level in \([0,3/8]\) lies in \([0,C/\sqrt{k}]\), once \(k>C^2\).  The
argument below will also prove strict increase on this interval, after
which these generalised inverses are ordinary inverses.

It remains to lower-bound the binomial atom in
\eqref{eq:Bk-derivative}.  Fix \(0\leq a\leq C/\sqrt{k}\).  Since
\(r_k(a)\) is the first point in a lattice of spacing two strictly above
\(kh(a)\), write
\[
 r_k(a)=kh(a)+u,
 \qquad 0<u\leq2.
\]
Put
\[
 n=k-1,
 \qquad p=\frac{1+a}{2},
 \qquad j=j_k(a)-1.
\]
For $k>C^2$, the bounds on $h$ give $0\leq h(a)<1$.  Hence
$0<r_k(a)\leq k$, so $1\leq j_k(a)\leq k$ and
$j\in\{0,\ldots,n\}$, as required in \Cref{lem:binomial-atom}.
The mean of the binomial variable in \eqref{eq:Bk-derivative} is \(np\).
Its displacement from the atom index satisfies the exact identity
\begin{align}
 2(j-np)
 &=2\left\{j_k(a)-1-(k-1)\frac{1+a}{2}\right\}\notag=k+r_k(a)-2-(k-1)(1+a)\notag\\
 &=r_k(a)-1-(k-1)a\notag\\
 &=k\{h(a)-a\}+(u-1+a).
 \label{eq:derivative-atom-displacement}
\end{align}
The previously proved bounds \(a/2\leq h(a)<a\) imply
\[
 0\leq a-h(a)\leq\frac a2.
\]
Also, \(0<u\leq2\) and, for \(k\geq C^2\),
\(0\leq a\leq C/\sqrt{k}\leq1\), so \(|u-1+a|\leq2\).  Taking absolute
values in \eqref{eq:derivative-atom-displacement} gives
\begin{align*}
 2|j-np|
 &\leq k|h(a)-a|+|u-1+a| \leq\frac{ka}{2}+2 \leq\left(\frac C2+2\right)\sqrt{k}.
\end{align*}
Consequently,
\begin{equation}
 |j-np|
 \leq\left(\frac C4+1\right)\sqrt{k}.
 \label{eq:derivative-atom-near-mean}
\end{equation}

Lemma~\ref{lem:binomial-atom} is stated using \(\sqrt n=\sqrt{k-1}\).
For \(k\geq2\),
\[
 \sqrt{k}\leq\sqrt2\sqrt{k-1}=\sqrt2\sqrt n,
\]
so \eqref{eq:derivative-atom-near-mean} implies
\[
 |j-np|\leq A\sqrt n,
 \qquad
 A:=\sqrt2\left(\frac C4+1\right)=\frac{7\sqrt2}{4}.
\]
In addition, for all sufficiently large \(k\),
\(a\leq C/\sqrt{k}\leq1/3\), and hence
\[
 p=\frac{1+a}{2}\in[1/2,2/3]\subseteq[1/3,2/3].
\]
For this value of $A$, the proof of \Cref{lem:binomial-atom} permits
the choices
\[
 n_A=882,
 \qquad
 c_A=\sqrt{\frac2\pi}
      \exp\left\{-\frac23-\frac83A^2\right\}
     =\sqrt{\frac2\pi}e^{-17},
\]
where the last equality uses $A^2=49/8$.  Thus all hypotheses of that
lemma hold uniformly for $k\geq883$, and
\[
 \Prob\left\{
 \operatorname{Bin}\left(k-1,\frac{1+a}{2}\right)=j_k(a)-1
 \right\}
 \geq\frac{c_A}{\sqrt{k-1}}
 \geq\frac{c_A}{\sqrt{k}}.
\]
Substituting this estimate into \eqref{eq:Bk-derivative} gives, with
$\gamma=c_A/2>0$,
\begin{equation}
 B_k'(a)\geq\gamma\sqrt{k}
 \label{eq:Bk-derivative-lower}
\end{equation}
for almost every \(a\in[0,C/\sqrt{k}]\).

Because \(B_k\) is absolutely continuous, the Fundamental Theorem of
Calculus for absolutely continuous functions and
\eqref{eq:Bk-derivative-lower} imply that, for
\(0\leq x<y\leq C/\sqrt{k}\),
\begin{align*}
 B_k(y)-B_k(x)
 &=\int_x^y B_k'(a)\,da\\
 &\geq\gamma\sqrt{k}(y-x)>0.
\end{align*}
Thus \(B_k\) is strictly increasing on \([0,C/\sqrt{k}]\).  Together with
continuity, \(B_k(0)=0\), and \(B_k(C/\sqrt{k})>3/8\), this proves that
the generalised inverse is an ordinary inverse on the relevant range and
that
\[
 B_k(B_k^{-1}(b))=b,
 \qquad 0\leq b\leq3/8.
\]

Finally, let \(0\leq b_0<b_1\leq3/8\) and put
\(\alpha_i=B_k^{-1}(b_i)\).  Both inverse values lie in
\([0,C/\sqrt{k}]\), so the preceding integrated derivative bound gives
\[
 b_1-b_0
 =B_k(\alpha_1)-B_k(\alpha_0)
 \geq\gamma\sqrt{k}(\alpha_1-\alpha_0).
\]
Rearranging and taking
$L=\gamma^{-1}=2/c_A=\sqrt{2\pi}e^{17}$ proves
\eqref{eq:Bk-inverse-lipschitz}.  Since strict increase and \(B_k(0)=0\)
give \(B_k^{-1}(0)=0\), setting \(b_0=0\) proves
\eqref{eq:Bk-inverse-size} and completes the proof.
\end{proof}

\begin{proof}[Proof of \Cref{lem:delayed-register-mixing-scales}]
Fix a delayed-register kernel and write \(u=jk+s\), where
\(j\in\mathbb Z_+\) and \(0\leq s<k\). If
\(\rho_s(x):=(x_{s+1},\ldots,x_k,x_1,\ldots,x_s)\), then the register
at time \(u\), started from \(x\), has the form
\(\rho_s(x)\odot W_{j,s}\), where \(\odot\) denotes coordinatewise
multiplication and the law of \(W_{j,s}\) does not depend on \(x\).
The map \(z\mapsto\rho_s(x)\odot z\) is a bijection of
\(\{-1,+1\}^k\) preserving \(U_k\). Hence the distance from \(U_k\)
is the same for every initial word, and therefore
\begin{equation}
 \max_x\left\|(P_{k,a}^{\mathrm{dr}})^u(x,\cdot)-U_k\right\|_{\TV}
 =\beta_a(u).
 \label{eq:register-worst-average-equality}
\end{equation}
Before \(k\), \eqref{eq:register-preflush} keeps this common distance
above \(1/2\), so $\tmix\geq k$. At time \(k\), the chain $P_{0}^{\mathrm{dr}}$ crosses $1/4$, since $a_0=B_k^{-1}\{\xi(1-\epsilon)\}$, which by its crossing-time definition \eqref{eq:Bk-inverse}, implies that
\(\beta_{a_0}(k)=\xi(1-\epsilon)\leq1/4\) (since $\xi\leq 1/4$ and $\epsilon>0$). For the other chain, $P_1^{\mathrm{dr}}$,
\(\beta_{a_1}(2k)\leq\xi(1-\epsilon)\leq1/4\). This proves the
claimed worst-case mixing-time bounds.

The worst-case and average mixing times agree at a common
threshold for the delayed register chain, as shown by \eqref{eq:register-worst-average-equality}.
The calculations following
\eqref{eq:register-block-identity} give
\(t_{\mathrm{mix},P_0^{\mathrm{dr}}}^{\sharp}(\xi)=k\) and
\(k<t_{\mathrm{mix},P_1^{\mathrm{dr}}}^{\sharp}(\xi)\leq2k\).
Since \(\xi\leq1/4\), threshold monotonicity and the preceding
worst-case bounds give the result.
\end{proof}

\begin{proof}[Proof of \Cref{lem:checkerboard-kernel}]

    Every entry of \(K_\theta\) is positive. In each checkerboard, the two signed perturbations in a paired row cancel, as do those in a paired column. Thus \(K_\theta\) has every row and column sum equal to one. Inspection of the two block rows and block columns in~\eqref{eq:rr-kernel} then shows that \(P_\theta\) is doubly stochastic. Its stationary law is therefore uniform. Moreover, one refresh already reaches any state in the opposite region and two refreshes reach any state in the initial region. Since every state also has a positive self-loop, \(P_\theta\) is uniformly ergodic.
\end{proof}

\begin{proof}[Proof of \Cref{lem:common-average-mixing-coefficient}]

Recall the rare-refresh checkerboard transition matrix $P_\theta$ in \eqref{eq:rr-kernel} with left-to-right kernel $K_\theta$ from \eqref{eq:rr-payload-block}. To isolate the parameter-bearing part, write
\begin{equation}
 C_\theta:=K_\theta-U_d,
 \qquad
 S_\theta:=\begin{pmatrix}0&K_\theta\\U_d&0\end{pmatrix}.
 \label{eq:rr-C-S}
\end{equation}
The matrix \(C_\theta\) has zero row and column sums, so
\begin{equation}
 C_\theta U_d=U_dC_\theta=0,
 \qquad
 K_\theta U_d=U_dK_\theta=U_d.
 \label{eq:rr-annihilation}
\end{equation}
Block multiplication gives
\begin{equation}
 S_\theta^2=\operatorname{diag}(U_d,U_d).
 \label{eq:rr-two-refreshes}
\end{equation}
Because the identity commutes with \(S_\theta\), the binomial theorem applied to \(P_\theta=(1-q)I_D+qS_\theta\) gives
\[
 P_\theta^t
 =\sum_{j=0}^t\binom tj(1-q)^{t-j}q^jS_\theta^j.
\]
Thus every term with \(j\geq2\) is parameter-free. Let \(P_{\boldsymbol 0}\) be the same rare-refresh chain with every block parameter set to zero. The term with \(j=1\) has weight \(w_t\), and hence
\begin{equation}
 P_\theta^t=P_{\boldsymbol 0}^t+w_t
 \begin{pmatrix}0&C_\theta\\0&0\end{pmatrix}.
 \label{eq:rr-parameter-term}
\end{equation}
Equation~\eqref{eq:rr-parameter-term} describes \(P_\theta^t(x,\cdot)\) after all intermediate state labels have been marginalised. A fully observed trajectory retains those labels and can contain several informative left-to-right refreshes.

Now subtract stationarity before taking absolute values. In the coefficient-zero chain, the only positive deviation from stationarity is at the initial state. All deviations in a probability row sum to zero, so the raw \(\ell^1\)-distance is \(2b_t\) and the row total variation is \(b_t\). In a left-starting row, a block of amplitude \(v\) produces the paired deviations
\begin{equation}
 \frac{w_tv-p_t/2}{d},
 \qquad
 \frac{-w_tv-p_t/2}{d}.
 \label{eq:rr-block-deviations}
\end{equation}
Combining these signed deviations before taking absolute values gives the increment \(g_t(v)/d\). Every right-starting row remains at \(b_t\). Stationary averaging over the two source rows associated with each block gives the factor \(4/D^2\), while maximising the nonnegative increments along a left source row gives the factor \(2/D\).
Summing these block contributions proves~\eqref{eq:rr-average-profile}.
\end{proof}

As a uniform consequence, $0\leq g_t(v)\leq w_t$ for $|v|<1$. Since there are $D^2/16$ blocks in the stationary average,
\begin{equation}
 b_t\leq\beta_\theta(t)\leq b_t+\frac{w_t}{4}.
 \label{eq:rr-profile-envelope}
\end{equation}

\begin{proof}[Proof of \Cref{lem:checkerboard-mixing-scales}]
For a rare-refresh checkerboard kernel \(P=P_\theta\), with
\(\theta\in(-1,1)^{2M}\), write
\(d_\theta(t):=\max_x\|P_\theta^t(x,\cdot)-\pi\|_{\TV}\). The row
calculation in Lemma~\ref{lem:common-average-mixing-coefficient} gives
\(b_t\leq d_\theta(t)\leq b_t+w_t/2\). At \(t=R/2\),
\(r_t\geq e^{-50/99}\); at \(t=4R\),
\(r_t\leq e^{-4}\), \(p_t\leq e^{-8}\), and
\(w_t\leq4e^{-3.99}\). Hence the envelope gives
\begin{align*}
    d_\theta(R/2)&\geq\frac{31}{32}e^{-50/99}>\frac14,\\
    d_\theta(4R)&\leq e^{-4}+\frac{e^{-8}}{64}+2e^{-3.99}<\frac14.
\end{align*}

Total-variation contraction proves the worst-case bounds.

The calculation proving \eqref{eq:rr-theorem-target-scale} gives \(2R<t_{\mathrm{mix},P}^{\sharp}(1/20)\leq10R\). Moreover,
\(\beta_P(t)\leq d_{\rm max}(t)\) and the submultiplicative
worst-case bound give
\begin{align*}
    t_{\mathrm{mix},P}^{\sharp}(1/20) &\leq t_{\mathrm{mix},P}(1/20) \leq\left\lceil\log_2(20)\right\rceil t_{\mathrm{mix},P}(1/4)\\
    &=5t_{\mathrm{mix},P}(1/4).
\end{align*}
Combining these inequalities with the worst-case bounds above proves the result.
\end{proof}

\begin{proof}[Proof of \Cref{lem:rr-confidence-calibration}]
Lemma~\ref{lem:rr-confidence-lag-bounds} gives $0<p<0.0014$, $w>0.117$ and $b\in (0.0344875,0.037)$.
Define the desired threshold levels \(L_0:=\xi_0-\eta\) and
\(L_1:=\xi_0+21\eta/20\), and set
\begin{equation}
 \begin{aligned}
 c_0&=\frac{4(\xi_0-\eta-b)+p/2}{w},\\
 c_1&=\frac{4(\xi_0+21\eta/20-b)+p/2}{w}.
 \end{aligned}
 \label{eq:rr-confidence-coefficients}
\end{equation}
Since \(\eta=\xi_0\epsilon=\epsilon/20\), and \(\epsilon\leq1/4\),
\[
 L_0=\xi_0(1-\epsilon)\geq0.0375>b.
\]
Equation~\eqref{eq:rr-confidence-coefficients} therefore gives
\(c_0-p/(2w)=4(L_0-b)/w>0\). Therefore, $g_{s_{0,1}}(c_0)>0$. Subtracting the two definitions in
\eqref{eq:rr-confidence-coefficients} gives
\[
 c_1-c_0=\frac{41\eta}{5w}>0,
\]
so $c_1>c_0>0$ also satisfies $g_{s_{0,1}}(c_1)>0$. Also,
\(L_1\leq0.063125\), and~\eqref{eq:rr-confidence-lag-bounds} gives
\[
 \begin{aligned}
 4(L_1-b)+\frac p2
 &<4(0.063125-0.0344875)+0.0007\\
 &=0.11525<0.99(0.117)<0.99w.
 \end{aligned}
\]
Thus \(c_1<0.99\), while \(w>0.117\) gives
\[
 c_1-c_0<\frac{41}{5(0.117)}\eta<71\eta.
\]
Since \(p>0\), we have \(0<c_0<c_1<0.99\), so both coefficients are admissible (i.e. lie in $(-1,1)$, as assumed). Hinge positivity (i.e., $wc_i>p/2$ for $i\in\{0,1\}$) gives
\(g_{s_{0,1}}(c_i)=wc_i-p/2\). Substituting
\eqref{eq:rr-confidence-coefficients} into~\eqref{eq:rr-common-profile} therefore yields~\eqref{eq:rr-confidence-profile-values}. The second value is strictly larger than \(\xi_0+\eta\), so monotonicity of $\beta_{c}(t)$ yields the stated crossing-time ordering and disjoint output sets.
\end{proof}

\begin{proof}[Proof of \Cref{lem:loc-hidden-coefficient-support}]
The accuracy condition, the bound $p\geq e^{-320/49}$ from
Lemma~\ref{lem:rr-localised-lag-bounds}, and
$\gamma_\ell\geq a_\star/\ell$ give
$\eta\leq c_{\rm acc}/\ell\leq p\gamma_\ell/128$, and hence
$$
h=32\eta/(w\gamma_\ell)\leq p/(4w)=\tau/2.
$$(Recall that the universal constants $\gamma_\ell,~a_\star$ and $c_{\rm acc}$ are defined in Appendix \ref{app:rr-constants}.) The same lemma gives
$0<\tau<0.00658$. Since $U=\tau+hZ$, $|Z|\leq1$ and $h\leq \tau/2$,
\[
 \frac{\tau}{2}\leq\tau-h\leq U\leq\tau+h
 \leq\frac{3\tau}{2}<\frac{3}{200}.
\]
This proves the claimed support inclusion for both translated priors.
\end{proof}

\begin{proof}[Proof of \Cref{lem:coefficient-validity-check}]
Since \(g_{s_{0,2}}(\tau+hZ)=whZ_+\) and \(0\leq Z_+\leq1\),
\[
 0\leq\frac{\bar g_0+\bar g_1}{2}\leq wh
 \leq\frac{w\tau}{2}=\frac p4.
\]
Lemma~\ref{lem:rr-localised-lag-bounds} then shows that~\eqref{eq:loc-control-calibration} has a solution on the positive linear branch satisfying \(0.556<c<0.753\). Together with Lemma~\ref{lem:loc-hidden-coefficient-support}, this shows that the coefficients do lie within the constructed classes ($\mathcal{U}_i$).
\end{proof}

\begin{proof}[Proof of \Cref{lem:loc-profile-concentration}]

Under prior $\mu_i$, the random variables $g_{s_{0,2}}(U_e)$ are independent, $[0,wh]$-valued, and have mean $\bar g_i$. Therefore, the (one-sided) Hoeffding's inequality gives,

\begin{align*}
   \Prb_{\mu_0^{\otimes M}}(\mathcal{U}_0^c)
   &=
   \Prb_{\mu_0^{\otimes M}}\left\{
   \frac1M\sum_{e=1}^M g_{s_{0,2}}(U_e)-\bar g_0>8\eta
   \right\}\\
   &\leq
   \exp\left\{-\frac{2M(8\eta)^2}{(wh)^2}\right\},
\end{align*}
and
\begin{align*}
   \Prb_{\mu_1^{\otimes M}}(\mathcal{U}_1^c) &= \Prb_{\mu_1^{\otimes M}}\left\{
   \frac1M\sum_{e=1}^M g_{s_{0,2}}(U_e)-\bar g_1\leq-8\eta
   \right\}\\
   &\leq
   \exp\left\{-\frac{2M(8\eta)^2}{(wh)^2}\right\}.
\end{align*}
Then, inserting \eqref{eq:loc-translated-prior}, and recalling the constant bounds in \Cref{asm:2} and \eqref{eq:loc-unit-hinge-gap},  yields \(2M(8\eta)^2/(wh)^2=M\gamma_\ell^2/8\geq\log80\). Hence \(\Prb_{\mu_i^{\otimes M}}(\mathcal{U}_i^c)\leq1/80\) for both \(i\in\{0,1\}\).

For \(u\in\mathcal U_0\), equations~\eqref{eq:loc-split-profile} and~\eqref{eq:loc-control-calibration} give
\[
 \beta_{(c,u)}(s_{0,2})
 \leq \E_0\beta_{(c,\boldsymbol U)}(s_{0,2})+\eta
 =\xi_0-\eta=\frac{1-\epsilon}{20}.
\]
Similarly, every \(u\in\mathcal U_1\) satisfies
\[
 \beta_{(c,u)}(s_{0,2})
 >\E_1\beta_{(c,\boldsymbol U)}(s_{0,2})-\eta
 =\xi_0+\eta=\frac{1+\epsilon}{20}.
\]

\end{proof}

\section{Supporting Lemmas}

\begin{lemma}
\label{lem:confidence-log-refinement}
$\log\frac1{4\delta}\geq\frac13\log\frac4\delta\quad\forall \quad0<\delta\leq\frac{1}{16}$.
\end{lemma}

\begin{proof}
Since $\log x > 0$ whenever $x>1$, we have
\[
 3\log\frac1{4\delta}-\log\frac4\delta
 =\log\frac{1}{256\delta^2}=2\log\frac1{16\delta}\geq0.
\]
\end{proof}

\begin{lemma}[Binomial upper-tail representation of total variation]
\label{lem:h-equivalence}
Fix $k\in\mathbb N$.  For $0\leq a<1$, let $P_a$ be the product law of
$k$ Rademacher variables $Z_1,\ldots,Z_k$ of mean $a$, let
$P_0=U_k$, and put
\[
 S=\sum_{i=1}^k Z_i,
 \qquad X=\frac{k+S}{2}.
\]
For $0<a<1$, define
\[
 h(a):=\frac{-\log(1-a^2)}{\log((1+a)/(1-a))},
\]
and set $h(0):=0$.  Then $h$ is continuous and strictly increasing on
$[0,1)$, and
\[
 \frac a2\leq h(a)<a,
 \qquad 0<a<1.
\]
For $0\leq a<1$, let $r_k(a)$ be the smallest point of
$\{-k,-k+2,\ldots,k\}$ strictly larger than $kh(a)$, and put
$j_k(a)=(k+r_k(a))/2$.  For $0<a<1$,
\[
 \left\{\frac{dP_a}{dP_0}>1\right\}
 =\{S>kh(a)\}
 =\{S\geq r_k(a)\}
 =\{X\geq j_k(a)\}.
\]
Moreover, for $0\leq a<1$, with $p=(1+a)/2$,
\begin{equation}
 \begin{aligned}
 B_k(a)
 &=\|P_a-P_0\|_{\TV}\\
 &=\Prob\{\operatorname{Bin}(k,p)\geq j_k(a)\}
   -\Prob\{\operatorname{Bin}(k,1/2)\geq j_k(a)\}.
 \end{aligned}
 \label{eq:Bk-binomial-tail}
\end{equation}
\end{lemma}

\begin{proof}
If a realisation has sum $S$, then it contains $(k+S)/2$ plus signs
and $(k-S)/2$ minus signs.  Independence therefore gives
\[
 \frac{dP_a}{dP_0}
 =(1+a)^{(k+S)/2}(1-a)^{(k-S)/2}.
\]
Taking logarithms and collecting the terms involving $k$ and $S$
yields
\[
 \log\frac{dP_a}{dP_0}
 =\frac{k}{2}\log(1-a^2)
  +\frac S2\log\frac{1+a}{1-a}.
\]

We next derive the integral representation for $h$ used
below.  Direct differentiation gives
\begin{align*}
 \frac{d}{du}\{-\log(1-u^2)\}
 &=-\frac{-2u}{1-u^2}
 =\frac{2u}{1-u^2},\\
 \frac{d}{du}\log\frac{1+u}{1-u}
 &=\frac{d}{du}\{\log(1+u)-\log(1-u)\}\\
 &=\frac1{1+u}+\frac1{1-u}
 =\frac{2}{1-u^2}.
\end{align*}
Both logarithmic expressions vanish at \(u=0\).  The Fundamental Theorem
of Calculus therefore gives
\[
 -\log(1-a^2)=\int_0^a\frac{2u}{1-u^2}\,du,
 \quad
 \log\frac{1+a}{1-a}=\int_0^a\frac{2}{1-u^2}\,du,
\]
and hence
\[
 h(a)=
 \frac{\displaystyle\int_0^a u\,w(u)\,du}
      {\displaystyle\int_0^a w(u)\,du},
 \qquad
 w(u):=\frac{2}{1-u^2}.
\]
Thus \(h(a)\) is the weighted average of \(u\) over \([0,a]\), with
positive increasing weight \(w\).  Since \(u\leq a\) and \(w(u)>0\),
\[
 u w(u)\leq a w(u),
\]
with strict inequality for \(0\leq u<a\).  Therefore
\[
 \int_0^a u w(u)\,du
 <a\int_0^a w(u)\,du.
\]
Dividing by the positive total weight \(\int_0^a w(u)\,du\) proves that
\(h(a)<a\).

We next prove that this weighted average is at least the midpoint \(a/2\).
Since \(\int_0^a w(u)\,du>0\), the desired inequality is equivalent to
\begin{align*}
 h(a)\geq\frac a2
 &\quad\Longleftrightarrow\quad
 \int_0^a u w(u)\,du
 \geq\frac a2\int_0^a w(u)\,du\\
 &\quad\Longleftrightarrow\quad
 \int_0^a\left(u-\frac a2\right)w(u)\,du\geq0.
\end{align*}
Thus it remains to prove that the last centred integral is nonnegative.
Split it at \(a/2\):
\begin{align*}
 I
 &:=
 \int_0^a\left(u-\frac a2\right)w(u)\,du=
 \int_0^{a/2}\left(u-\frac a2\right)w(u)\,du
 +\int_{a/2}^a\left(u-\frac a2\right)w(u)\,du.
\end{align*}
The first integral is nonpositive and the second is nonnegative.  To compare
their magnitudes, make the change of variables \(v=a-u\) in the second
integral.  Then \(u=a-v\), \(du=-dv\), and the transformed endpoints are
\(v=a/2\) when \(u=a/2\) and \(v=0\) when \(u=a\).  Hence
\begin{align*}
 \int_{a/2}^a\left(u-\frac a2\right)w(u)\,du
 &=
 -\int_{a/2}^{0}\left(\frac a2-v\right)w(a-v)\,dv\\
 &=
 \int_0^{a/2}\left(\frac a2-v\right)w(a-v)\,dv\\
 &=
 \int_0^{a/2}\left(\frac a2-u\right)w(a-u)\,du.
\end{align*}
The second equality uses
\(\int_b^c f(v)\,dv=-\int_c^b f(v)\,dv\): the minus sign from
\(du=-dv\) cancels the minus sign introduced by reversing the endpoints.
The final equality merely renames the dummy variable \(v\) as \(u\).
Since
\[
 \int_0^{a/2}\left(u-\frac a2\right)w(u)\,du
 =
 -\int_0^{a/2}\left(\frac a2-u\right)w(u)\,du,
\]
the two halves combine to give
\begin{align*}
 I
 &=
 \int_0^{a/2}\left(\frac a2-u\right)
       \{w(a-u)-w(u)\}\,du\geq0,
\end{align*}
because \(a-u\geq u\) on \([0,a/2]\) and \(w\) is increasing.  By the
equivalence above, \(I\geq0\) proves \(h(a)\geq a/2\).

Finally, write
\[
 N(a):=\int_0^a u w(u)\,du,
 \qquad
 D(a):=\int_0^a w(u)\,du,
\]
so that \(h(a)=N(a)/D(a)\).  The variable \(a\) appears in these
integrals through their upper endpoints.  The Fundamental Theorem of
Calculus therefore gives
\[
 N'(a)=a w(a),
 \qquad
 D'(a)=w(a).
\]
Since \(D(a)>0\), the quotient rule yields
\begin{align*}
 h'(a)
 &=
 \frac{N'(a)D(a)-N(a)D'(a)}{D(a)^2} =
 \frac{a w(a)D(a)-N(a)w(a)}{D(a)^2}\\
 &=
 \frac{w(a)}{D(a)}
 \left\{a-\frac{N(a)}{D(a)}\right\} =
 \frac{w(a)\{a-h(a)\}}{\int_0^a w(u)\,du}.
\end{align*}
Here \(w(a)>0\), \(D(a)>0\), and the previously proved inequality
\(h(a)<a\) gives \(a-h(a)>0\).  Consequently \(h'(a)>0\) for \(a>0\).
The defining quotient is continuous on $(0,1)$, while the bounds above
give $h(a)\to0=h(0)$ as $a\downarrow0$.  Thus $h$ is continuous on
$[0,1)$.  Together with $h'(a)>0$ for $a>0$ and
$h(a)>0=h(0)$, this proves that $h$ is strictly increasing and that
\[
 \frac a2\leq h(a)<a,
 \qquad 0<a<1.
\]

The definition of \(h\) makes its likelihood-ratio role explicit.  For
\(a>0\), the preceding formula for the log likelihood ratio can be
rewritten as
\[
 \log\frac{dP_a}{dP_0}
 =
 \frac12\log\frac{1+a}{1-a}\{S-kh(a)\}.
\]
Indeed, after multiplying out, the definition
\[
 h(a)=\frac{-\log(1-a^2)}{\log((1+a)/(1-a))}
\]
turns the term involving \(h(a)\) into
\((k/2)\log(1-a^2)\).  The factor outside the braces is positive, so
\[
 \frac{dP_a}{dP_0}>1
 \quad\Longleftrightarrow\quad
 S>kh(a).
\]
Since $S$ takes values in the lattice $\{-k,-k+2,\ldots,k\}$, the
definitions of $r_k(a)$, $j_k(a)$, and $X$ give
\[
 \{S>kh(a)\}=\{S\geq r_k(a)\}=\{X\geq j_k(a)\}.
\]
Thus the event on which the likelihood ratio exceeds one is a binomial
upper-tail event, namely $\{X\geq j_k(a)\}$.

For two probability mass functions $p$ and $q$ on a finite set,
\[
 \|p-q\|_{\TV}
 =\sum_{x:p(x)>q(x)}\{p(x)-q(x)\}.
\]
Because $P_0$ has full support, the event on which the likelihood ratio
exceeds one collects exactly the positive terms in the signed mass difference
and realises total variation.  Since the definition of $B_k$ gives
$B_k(a)=\|P_a-P_0\|_{\TV}$, it follows that
\[
 B_k(a)
 =P_a\{S>kh(a)\}-P_0\{S>kh(a)\}.
\]
Under $P_a$, $X\sim\operatorname{Bin}(k,(1+a)/2)$, whereas under
$P_0$, $X\sim\operatorname{Bin}(k,1/2)$.  The preceding event
equivalences therefore prove~\eqref{eq:Bk-binomial-tail} for $a>0$.
At $a=0$, both binomial upper-tail probabilities are equal and $B_k(0)=0$,
so the identity also holds at that endpoint.
\end{proof}

\begin{lemma}[Derivative of $B_k$]
\label{lem:binomial-equivalent-derivative}
For every $k\in\mathbb N$, $B_k$ is locally absolutely continuous on $[0,1)$.
At every $a\in(0,1)$ for which $r_k(a)$, equivalently $j_k(a)$, is
locally constant,
\begin{equation}
 B_k'(a)
 =\frac{k}{2}
 \Prob\left\{
 \operatorname{Bin}\left(k-1,\frac{1+a}{2}\right)=j_k(a)-1
 \right\}.
 \label{eq:Bk-derivative}
\end{equation}
Consequently, \eqref{eq:Bk-derivative} holds for almost every
$a\in(0,1)$.
\end{lemma}

\begin{proof}[Proof of \Cref{lem:binomial-equivalent-derivative}]
Put $p=(1+a)/2$, and write $P_p$ and $\E_p$ for probability and
expectation when $X\sim\operatorname{Bin}(k,p)$.  By
\eqref{eq:Bk-binomial-tail}, $B_k(a)$ is the difference of the two
corresponding binomial upper-tail probabilities.

Because \(h\) is strictly increasing, the lattice threshold \(r_k(a)\)
is piecewise constant and changes at only finitely many values of \(a\).
Fix an open interval on which \(r_k(a)\), and hence \(j_k(a)\), has a
constant value \(j\).  On this interval, the second term in
\eqref{eq:Bk-binomial-tail} is constant.  Term-by-term differentiation of
the first, finite binomial sum gives
\begin{align}
 \frac{d}{dp}P_p\{X\geq j\}
 &=\sum_{x=j}^k
   \frac{d}{dp}\left\{\binom{k}{x}p^x(1-p)^{k-x}\right\}\notag \\
   &=\sum_{x=j}^k\binom{k}{x}p^x(1-p)^{k-x}
   \left\{\frac{x}{p}-\frac{k-x}{1-p}\right\}\notag\\
   &=\frac{1}{p(1-p)}
   \E_p[(X-kp)\one\{X\geq j\}].
 \label{eq:binomial-tail-score}
\end{align}
To evaluate the final expectation, the binomial identities
\[
 x\binom{k}{x}=k\binom{k-1}{x-1},
 \qquad
 (k-x)\binom{k}{x}=k\binom{k-1}{x}
\]
give
\begin{align*}
 \E_p[X\one\{X\geq j\}]
 &=kp\Prob\{\operatorname{Bin}(k-1,p)\geq j-1\},\\
 \E_p[(k-X)\one\{X\geq j\}]
 &=k(1-p)\Prob\{\operatorname{Bin}(k-1,p)\geq j\}.
\end{align*}
Since \(X-kp=(1-p)X-p(k-X)\), subtracting the two adjacent tails leaves
one atom:
\[
 \E_p[(X-kp)\one\{X\geq j\}]
 =kp(1-p)\Prob\{\operatorname{Bin}(k-1,p)=j-1\}.
\]
Substitution into \eqref{eq:binomial-tail-score} therefore yields
\[
 \frac{d}{dp}P_p\{X\geq j\}
 =k\Prob\{\operatorname{Bin}(k-1,p)=j-1\}.
\]
Finally, $dp/da=1/2$.  The chain rule gives
\eqref{eq:Bk-derivative} whenever the lattice threshold is locally
constant.  There are only finitely many exceptional points at which the threshold changes.  Moreover, from its
definition as total variation, \(B_k\) is a finite sum of absolute values
of analytic functions of $a$.  It is therefore continuous and locally
absolutely continuous on $[0,1)$.  Thus \eqref{eq:Bk-derivative} holds
almost everywhere.
\end{proof}

\begin{lemma}[Finite binomial atom bound]
\label{lem:binomial-atom}
For every finite \(A>0\), there exist constants \(c_A>0\) and
\(n_A\in\mathbb N\) such that, whenever \(n\geq n_A\),
\(p\in[1/3,2/3]\), and \(j\in\{0,\ldots,n\}\) satisfy
\(|j-np|\leq A\sqrt n\),
\[
 \Prob\{\operatorname{Bin}(n,p)=j\}
 \geq\frac{c_A}{\sqrt n}.
\]
\end{lemma}

\begin{proof}
Put \(q=j/n\).  The assumption
\(|j-np|\leq A\sqrt n\) implies that
\[
 |q-p|\leq\frac{A}{\sqrt n}.
\]
Choose
\[
 n_A=\left\lceil\max\{4,144A^2\}\right\rceil.
\]
If \(n\geq n_A\), then
\[
 |q-p|\leq\frac{A}{\sqrt n}\leq\frac1{12}.
\]
Since \(p\in[1/3,2/3]\), it follows that \(q\in[1/4,3/4]\).
Thus \(j\) and \(n-j\) are both positive and of order \(n\), and the
Stirling and Taylor bounds below hold uniformly.  In particular,
\(j,n-j\geq n/4\geq1\), so Robbins' formula applies to all three
factorials in \(n!/[j!(n-j)!]\) = $\binom{n}{j}$.

For every positive integer \(m\), write Robbins' formula as
\[
 m!=\sqrt{2\pi m}\left(\frac me\right)^m e^{r_m},
 \qquad
 \frac1{12m+1}<r_m<\frac1{12m}.
\]
Substitution into the binomial coefficient gives
\begin{align*}
 \binom nj
 &=
 \frac{1}{\sqrt{2\pi}}
 \sqrt{\frac{n}{j(n-j)}}
 \frac{n^n}{j^j(n-j)^{n-j}}
 \exp\{r_n-r_j-r_{n-j}\}.
\end{align*}
Here the factors involving \(e\) cancel because
\[
 e^{-n}e^je^{n-j}=1,
\]
while the three Robbins correction factors collect into
\(\exp\{r_n-r_j-r_{n-j}\}\).  For the square-root factor, the identities
\(j=qn\) and \(n-j=(1-q)n\) give
\[
 \frac{1}{\sqrt{2\pi}}\sqrt{\frac{n}{j(n-j)}}
 =
 \frac{1}{\sqrt{2\pi n q(1-q)}}.
\]
For the power factor, taking logarithms and making the same substitution
shows how the binary entropy arises:
\begin{align*}
 \log\left(\frac{n^n}{j^j(n-j)^{n-j}}\right)
 &=n\log n-j\log j-(n-j)\log(n-j)\\
 &=n\log n-qn\log(qn)-(1-q)n\log((1-q)n)\\
 &=n\log n-qn\{\log q+\log n\} -(1-q)n\{\log(1-q)+\log n\}\\
 &=-qn\log q-(1-q)n\log(1-q)\\
 &=nH(q),
\end{align*}
where \(H(q):=-q\log q-(1-q)\log(1-q)\).  Exponentiating therefore gives
\[
 \frac{n^n}{j^j(n-j)^{n-j}}=\exp\{nH(q)\}.
\]
Combining the square-root, power, and remainder factors yields
\[
 \binom nj
 =
 \frac{1}{\sqrt{2\pi n q(1-q)}}
 \exp\{nH(q)\}
 \exp\{r_n-r_j-r_{n-j}\}.
\]
Since \(q\in[1/4,3/4]\),
\[
 j=nq\geq\frac n4,
 \qquad
 n-j=n(1-q)\geq\frac n4.
\]
Moreover, \(r_n>0\), while the upper bound in Robbins' formula gives
\[
 r_j<\frac1{12j},
 \qquad
 r_{n-j}<\frac1{12(n-j)}.
\]
Consequently,
\begin{align*}
 r_n-r_j-r_{n-j}
 &>-\frac1{12j}-\frac1{12(n-j)}\geq-\frac1{3n}-\frac1{3n}=-\frac{2}{3n}\geq-\frac23.
\end{align*}
Also \(q(1-q)\leq1/4\), and hence
\begin{align*}
 \frac{1}{\sqrt{2\pi n q(1-q)}}
 &\geq\frac{1}{\sqrt{2\pi n(1/4)}}=\sqrt{\frac{2}{\pi n}}.
\end{align*}
Multiplying the preceding formula for \(\binom nj\) by
\(p^j(1-p)^{n-j}\), and using the exact identity
\begin{align*}
 H(q)+q\log p+(1-q)\log(1-p)
 &=-q\log\frac qp-(1-q)\log\frac{1-q}{1-p}=-\operatorname{kl}(q\Vert p),
\end{align*}
where \(\operatorname{kl}(q\Vert p)\) is the one-trial Bernoulli
relative entropy, therefore gives
\[
 \Prob\{\operatorname{Bin}(n,p)=j\}
 \geq\sqrt{\frac2\pi}\,\frac1{\sqrt n}
       \exp\left\{-\frac23-n\operatorname{kl}(q\Vert p)\right\}.
\]

For fixed \(p\), let \(f_p(x)=\operatorname{kl}(x\Vert p)\).  Then for all $x\in[1/4,3/4]$,
\[
 f_p(p)=f_p'(p)=0,
 \qquad
 f_p''(x)=\frac1{x(1-x)}\leq\frac{16}{3}.
\]
The segment between \(p\) and \(q\) lies in this interval, so Taylor's
theorem gives
\[
 \operatorname{kl}(q\Vert p)
 \leq\frac83(q-p)^2.
\]
Consequently \(n\operatorname{kl}(q\Vert p)\leq8A^2/3\), and the claimed
bound follows, for instance, with
\[
 c_A=\sqrt{\frac2\pi}
      \exp\left\{-\frac23-\frac83A^2\right\}.
\]
\end{proof}

\begin{lemma}
\label{lem:rr-confidence-lag-bounds}
At \(s_{0,1}=33R/10\), the quantities defined in
\eqref{eq:rr-profile-scalars}--\eqref{eq:rr-baseline-hinge} satisfy
\begin{equation}
 \begin{aligned}
 0.0356&<r<0.037, & 0&<p<0.0014,\\
 0.117&<w<0.124, & 0.0344875&<b<0.037.
 \end{aligned}
 \label{eq:rr-confidence-lag-bounds}
\end{equation}
\end{lemma}

\begin{proof}
For \(0<x<1\), the Fundamental Theorem of Calculus gives
\[
 \log(1-x)=-\int_0^x\frac{du}{1-u},
\]
and bounding the integrand on \([0,x]\) yields
\[
 -\frac{x}{1-x}\leq\log(1-x)\leq-x.
\]
Applying these bounds with \(x=1/R\), and using
\(R/(R-1)\leq100/99\), gives
\[
 e^{-10/3}\leq r=(1-1/R)^{33R/10}\leq e^{-3.3}.
\]
Since \(e^{-10/3}>0.0356\) and \(e^{-3.3}<0.037\), this proves the
stated bounds on \(r\). Applying the upper logarithmic bound with
\(x=2/R\) gives
\[
 0<p=(1-2/R)^{33R/10}\leq e^{-6.6}<0.0014.
\]
From the definition of \(w\),
\[
 w=\frac{33}{10}\frac{r}{1-1/R}.
\]
The bounds on \(r\) and \(1-1/R\geq0.99\) give
\(0.117<w<0.124\). Finally, \(p/2<r\), because
\(0<1-2/R<1-1/R\), and
\(1-2/D\geq31/32\) imply
\[
 \begin{aligned}
 b&=r\left(1-\frac2D\right)+\frac pD<r<0.037,\\
 b&>0.0356\frac{31}{32}=0.0344875.
 \end{aligned}
\]
\end{proof}

\begin{lemma}
\label{lem:rr-localised-lag-bounds}
Under Assumption~\ref{asm:1}, at $s_{0,2}=16R/5$, write
$r=r_{s_{0,2}}$, $p=p_{s_{0,2}}$, $w=w_{s_{0,2}}$,
$b=b_{s_{0,2}}$, and $\tau=p/(2w)$. Then
\begin{equation}
 \begin{aligned}
 e^{-320/99}&\leq r\leq e^{-16/5},\\
 e^{-320/49}&\leq p\leq e^{-32/5},\\
 \frac{16}{5}e^{-320/99}&\leq w\leq\frac{16}{5}e^{-319/100},\\
 0&<\tau<0.00658<10^{-2},\\
 \frac{31}{32}e^{-320/99}&\leq b\leq e^{-16/5}+\frac{e^{-32/5}}{64}.
 \end{aligned}
 \label{eq:rr-localised-lag-bounds}
\end{equation}
\end{lemma}

\begin{proof}
Apply the logarithmic inequalities used in the proof of
Lemma~\ref{lem:rr-confidence-lag-bounds}, now at $s_{0,2}=16R/5$,
using $R/(R-1)\leq100/99$ and $R/(R-2)\leq50/49$. This gives the
displayed bounds on $r$ and $p$. The remaining bounds follow from
\[
 w=\frac{16}{5}(1-1/R)^{16R/5-1},\qquad
 \tau=\frac{p}{2w},\qquad
 b=r\left(1-\frac2D\right)+\frac pD,
\]
using $R\geq100$ and $D\geq64$.
\end{proof}

\section{Constants for the rare-refresh checkerboard}
\label{app:rr-constants}

For every integer \(j\geq1\), define $\gamma_j:=\inf_{\deg(p)\leq j}
 \|\,|x|-p(x)\|_{\infty,[-1,1]}$ and $a_\star:=\inf_{j\geq1}j\gamma_j$. The finite-degree Bernstein bound gives
\(a_\star\geq1/\{4(1+\sqrt2)\}>0\). Fix the universal constants appearing in the formal results section: $c_{\rm acc}:=\frac{e^{-320/49}a_\star}{128},~ C_{\rm size}:=\frac{8\log80}{a_\star^2}$ and $c_{\rm inf}:=\left\{(16/5)e^{-320/99}\right\}^2a_\star^2(1-10^{-4})/(2^{15}e)$.

\end{document}